\newif{\ifarxiv}
\arxivtrue

\documentclass[12pt]{amsart}
\pdfoutput=1 
\usepackage[inline]{enumitem}
\usepackage{amssymb}
\usepackage{graphicx}
\usepackage{amsfonts}
\usepackage{amsmath}
\usepackage{verbatim}
\usepackage{url}
\usepackage{stmaryrd}
\usepackage{latexsym}
\usepackage[all]{xy}

\makeatletter
\@namedef{subjclassname@2010}{%
  \textup{2010} Mathematics Subject Classification}
\makeatother

\newtheorem{thm}{Theorem}[section]
\newtheorem{lemma}[thm]{Lemma}
\newtheorem{proposition}[thm]{Proposition}
\newtheorem{corollary}[thm]{Corollary}

\theoremstyle{definition}
\newtheorem{definition}[thm]{Definition}

\newtheorem{example}[thm]{Example}

\numberwithin{equation}{section}

\newcommand\Sober{\mathcal S}

\newcommand\nat{\mathbb{N}}

\newcommand\lat[2]{{#1} ({#2})}

\newcommand\dc{\mathop{\downarrow}\nolimits}
\newcommand\uuarrow{\rlap{$\uparrow$}\raise.5ex\hbox{$\uparrow$}}
\newcommand\ddarrow{\rlap{$\downarrow$}\raise.5ex\hbox{$\downarrow$}}

\newcommand\Img{\mathop{\mathrm{Im}}\nolimits}

\newcommand\eqdef{\mathrel{\buildrel \text{def}\over=}}
\newcommand\diff{\smallsetminus}

\newcommand\bm[1]{\mathtt{#1}}
\newcommand\bP{\bm{P}}
\newcommand\bQ{\bm{Q}}

\newcommand\bB{\bm{B}}

\newcommand\bA{\bm{A}}

\newcommand\adjsw[1]{{#1}^\circ}
\newcommand\flt[1]{[{#1}]}

\newcommand\rsob{\mathrm{rsob}\;}
\newcommand\supp[1]{\mathrm{supp}\;{#1}}
\newcommand\ls{\leq^{\mathrm{top}}_*}

\begin{document}


\baselineskip=17pt


\title[Infinitary Noetherian Constructions II]{Infinitary Noetherian
  Constructions II. Transfinite Words and the Regular Subword Topology}

\author[Goubault-Larrecq]{Jean Goubault-Larrecq}
\thanks{The first and third authors were supported by grant ANR-17-CE40-0028 of the
  French National Research Agency ANR (project BRAVAS)}
\address{Universit\'e Paris-Saclay, CNRS, ENS Paris-Saclay, Laboratoire M\'ethodes Formelles, 91190, Gif-sur-Yvette, France}
\email{goubault@lsv.fr}

\author[Halfon]{Simon Halfon}
\address{Universit\'e Paris-Saclay, CNRS, ENS Paris-Saclay, Centre
  Borelli, 91190, Gif-sur-Yvette, France.}
\email{simon.halfon@ens-paris-saclay.fr}

\author[Lopez]{Aliaume Lopez}
\address{Universit\'e Paris-Saclay, CNRS, ENS Paris-Saclay,
  Laboratoire M\'ethodes Formelles, 91190, Gif-sur-Yvette, France}
\address{Universit\'e de Paris, IRIF, CNRS, F-75013 Paris, France}
\email{aliaume.lopez@ens-paris-saclay.fr}

\date{}

\begin{abstract}
  We show that the spaces of transfinite words, namely ordinal-indexed
  words, over a Noetherian space, is also Noetherian, under a natural
  topology which we call the regular subword topology.  We
  characterize its sobrification and its specialization ordering, and
  we give an upper bound on its sobrification rank and on its stature.
\end{abstract}



\maketitle

\section{Introduction}
\label{sec:intro}

Given a well-quasi-order $X$ (wqo, for short), the space of infinite
words $X^\omega$ (of length $\omega$) need not be wqo in the subword
preordering.  One way of correcting this anomaly is to turn to the
stronger notion of better quasi-orderings \cite{NW:bqo}.  Another one
is to turn to the weaker notion of Noetherian space.  Noetherian
spaces are a natural, topological generalization of wqos with many
similar properties \cite[Section~9.7]{JGL-topology}.  For example,
there are Noetherian analogues of Higman's theorem and of Kruskal's
theorem.  Noetherianity is also preserved by some infinitary
constructions such as powerset.

In part~I of this work \cite{JGL:noethinf} we have shown that, given a
Noetherian space $X$, $X^\omega$ is again Noetherian, with a natural
topology, the subword topology.  The same works for the set of
finite-or-infinite words $X^{\leq\omega}$.  The purpose of the present
paper is to extend this to spaces $X^{<\alpha}$ of transfinite words,
namely words indexed by ordinals strictly smaller than a fixed ordinal
bound $\alpha$.

The topology we choose is a natural generalization of
that of part~I.
The bulk of the work consists in showing that if $X$ is Noetherian,
then so is $X^{<\alpha}$.

\paragraph{\emph{Outline.}}  Section~\ref{sec:preliminaries}
recapitulates a few basic notions, and is also where we state our
basic tool for showing that $X^{<\alpha}$ is Noetherian, as
Proposition~\ref{prop:S}.  We define the regular subword topology on
$X^{<\alpha}$ in Section~\ref{sec:transf-sequ}, based on so-called
transfinite products.  After a few basic results on transfinite words
in Section~\ref{sec:comb-transf-words}, we show that the
$\alpha$-products, a specific kind of transfinite products, form an
irredundant subbase of closed subsets of $X^{<\alpha}$, where $\alpha$
is a special kind of ordinal which we call a bound; this is one of the
conditions of Proposition~\ref{prop:S}.  We characterize inclusion of
transfinite products in Section~\ref{sec:incl-transf-prod}, and we
show that every transfinite product has a canonical
form 
in Section~\ref{sec:normal-forms}.  This allows us to show that the
collection of $\alpha$-products is well-founded in
Section~\ref{sec:well-foundedness}; this is another condition of
Proposition~\ref{prop:S}.  The final condition requires us to express
finite intersections of $\alpha$-products as finite unions of
$\alpha$-products, which we do in
Section~\ref{sec:inters-transf-prod}.  In the process, we obtain an
upper bound on the stature and reduced sobrification rank of
$X^{<\alpha}$.  Finally, we describe the specialization preordering of
$X^{<\alpha}$
in Section~\ref{sec:spec-order-xalpha}.

\section{Preliminaries}
\label{sec:preliminaries}

\subsection{Topology.}

Most of the following can be found in \cite{JGL-topology}.

Every topological space has a \emph{specialization} preordering
$\leq$, defined by $x \leq y$ if and only if every open neighborhood
of $x$ contains $y$.  The closure of $\{x\}$ is the \emph{principal
  ideal} $\dc x \eqdef \{y \in X \mid y \leq x\}$.  
We silently consider any topological space as a preordered set
under $\leq$.

The \emph{Alexandroff topology} of a preordering $\leq$ is its family
of upwards-closed sets.  Among the topologies with a given
specialization preordering $\leq$, it is the finest.  The coarsest is
the \emph{upper topology}, whose closed sets are intersections of sets
of the form $\dc E$, $E$ finite; we write $\dc E$ for
$\bigcup_{x \in E} \dc x$.


A \emph{Noetherian} space is a topological space in which every open
set is compact.  We do not require compactness to imply separation.

A strict partial ordering $<$ is \emph{well-founded} if and only if
there is no infinite strictly descending chain.
By a slight abuse of language, we say that a preordering $\leq$ is
well-founded if and only if its strict part ($x < y$ if and only if
$x \leq y$ and $y \not\leq x$) is well-founded. A space is Noetherian
if and only if its set of closed subsets is well-founded under
inclusion.

A closed subset $C$ 
is \emph{irreducible} if and only if $C \neq \emptyset$, and for all
closed sets $F_1$, $F_2$ such that $C \subseteq F_1 \cup F_2$, $C$ is
included in $F_1$ or in $F_2$.  The closure $\dc x$ of every point $x$
is irreducible closed.  The product $C_1 \times C_2$ of two
irreducible closed subsets is irreducible in the product topology.  A
space is \emph{sober} if and only if every irreducible closed subset
is the closure $\dc x$ of a unique point $x$.

An important property is that, in a Noetherian space $X$, every closed
subset is a \emph{finite} union of irreducible closed subsets.

The \emph{sobrification} $\Sober X$ of a topological space $X$ is its
set of irreducible closed subsets, with the sets
$\diamond U \eqdef \{C \in \Sober X \mid C \cap U \neq \emptyset\}$,
$U$ open in $X$, as open sets.  The specialization ordering of $\Sober X$
is inclusion.  $\Sober X$ is sober, and its lattice of open subsets is
order-isomorphic to that of $X$, through 
$U \mapsto \diamond U$.  In particular, $X$ is Noetherian if and only
if $\Sober X$ is.

$\Sober$ defines a functor: for every continuous map
$f \colon X \to Y$, $\Sober f$ maps every $C \in \Sober X$ to the
closure $cl (f [C])$ in $Y$ of $f [C]$.  (We write $f [C]$ for the
direct image $\{f (x) \mid x \in C\}$.)  In particular, $cl (f [C])$
is irreducible closed.

The sober Noetherian spaces can be characterized
order-theoretic\-ally: they are exactly the sets $X$ with a
well-founded preordering $\leq$ such that every finite intersection of
principal ideals can be expressed as a finite union of principal
ideals; the topology of $X$ is the upper topology of $\leq$. Then the
closed subsets are exactly the sets $\dc E$ with $E$ finite.

The following proposition refines this, and will be the core of our
constructions.  A family of sets $\mathcal P$ is \emph{irredundant} if
and only if every element of $\mathcal P$ is irreducible in
$\mathcal P$, namely: no element of $\mathcal P$ is empty, and for all
$P, P_1, P_2 \in \mathcal P$, if $P \subseteq P_1 \cup P_2$ then
$P \subseteq P_1$ or $P \subseteq P_2$.
\begin{proposition}
  \label{prop:S}
  Let $\mathcal P$ be a family of subsets of a set $X$, such that:
  \begin{enumerate}
  \item $\mathcal P$ is well-founded under inclusion;
  \item $X$ can be written as a finite union of elements of $\mathcal
    P$;
  \item for all $P, Q \in \mathcal P$, $P \cap Q$ is a
    finite union of elements of $\mathcal P$.
  \end{enumerate}
  Then $X$, with the coarsest topology that makes every element of
  $\mathcal P$ closed, is Noetherian.
  If $\mathcal P$ is irredundant, its irreducible closed subsets are
  exactly the elements of $\mathcal P$, and $\Sober X$ equals
  $\mathcal P$ with the upper topology of inclusion.
\end{proposition}
\begin{proof}
  By assumption, $\mathcal P$ is sober Noetherian in the upper
  topology of $\subseteq$.
  
  For every $x \in X$, since $\mathcal P$ is well-founded
  (property~(1)), there is a minimal element $P$ of $\mathcal P$ that
  contains $x$.  For every $Q \in \mathcal P$ that contains $x$, we
  can write $P \cap Q$ as $\bigcup_{i=1}^n P_i$ where each $P_i$ is in
  $\mathcal P$, by~(3).  Then $x$ is in one $P_i$, and by minimality
  of $P$, $P = P_i$.  It follows that $P = P_i \subseteq P \cap Q$, so
  $P \subseteq Q$.  Hence $P$ is the smallest element of $\mathcal P$
  that contains $x$.  Let us write that element as $\eta (x)$; for
  every $Q \in \mathcal P$, $\eta (x) \subseteq Q$ if and only if
  $x \in Q$.

  This defines a map $\eta \colon X \to \mathcal P$.  For every finite
  subset $E \eqdef \{P_1, \cdots, P_n\}$ of $\mathcal P$,
  $\eta^{-1} (\dc E) = \{x \in X \mid \exists i, \eta (x) \subseteq
  P_i\} = \bigcup_{i=1}^n P_i$ is closed in $X$. Hence $\eta$ is
  continuous.  Taking $n \eqdef 1$, we obtain that every element $P$
  of $\mathcal P$ can be written as $\eta^{-1} (\dc \{P\})$.  Since
  $\eta^{-1}$ commutes with all intersections and (finite) unions,
  $\eta$ is a \emph{full} map, viz.\ every closed subset of $X$ is the
  inverse image of some closed subset of $\mathcal P$---this was
  called an \emph{initial} map in \cite{JGL:noethinf}.  Lemma~8 of
  that paper states that any space from which there is a full map to a
  Noetherian space is itself Noetherian.  Therefore $X$ is Noetherian.

  By Lemma~9 of the same paper, every irreducible closed subset $C$ of
  $X$ is of the form $\eta^{-1} (\mathcal C)$ for some irreducible
  closed subset $\mathcal C$ of $\mathcal P$.  Since $\mathcal P$ is
  sober, $\mathcal C = \dc P$ for some unique $P \in \mathcal P$,
  hence $C = P$.


  Conversely, and assuming $\mathcal P$ irredundant, we claim that
  every element $P$ of $\mathcal P$ is irreducible in $X$.  By
  assumption, $P$ is non-empty.  Since $\eta$ is full, the closed
  subsets of $X$ are exactly the inverse images of closed subsets of
  $\mathcal P$, namely the finite unions of elements of $P$.  Hence it
  suffices to show that if $P$ is included in a finite union of
  elements of $\mathcal P$, then it is included in one of them.  This
  follows directly from the fact that $\mathcal P$ is irredundant.
\end{proof}

\subsection{Ordinals.}

An \emph{indecomposable} ordinal is one of the form $\omega^\beta$,
where $\beta$ is an ordinal; equivalently, an ordinal that cannot be
written as the sum of two strictly smaller ordinals.  The other
ordinals are \emph{decomposable}.  For all ordinals $\alpha$, $\beta$,
$\gamma$, we have:
\begin{enumerate*}
\item if $\gamma$ is indecomposable and
  $\alpha, \beta < \gamma$, then $\alpha+\beta < \gamma$;
\item if $\gamma$ is indecomposable,
  $\alpha < \gamma$, and $\beta \leq \gamma$, then
  $\alpha+\beta \leq \gamma$;
\item every ordinal $\alpha$ can be written in a unique way as a
  finite sum of indecomposable ordinals $\gamma_1 + \cdots + \gamma_m$
  (its \emph{Cantor normal form}), where
  $\alpha \geq \gamma_1 \geq \cdots \geq \gamma_m$;
\item the ordering on Cantor normal forms is lexicographic: $\gamma_1
  + \cdots + \gamma_m < \gamma'_1 + \cdots + \gamma'_n$ (where both
  sides are in Cantor normal form) if and only if for some $i \geq 1$,
  $\gamma_1=\gamma'_1$, \ldots,
  $\gamma_{i-1}=\gamma'_{i-1}$ and either $i=m+1 \leq n$, or $i \leq m, n$ and
  $\gamma_i < \gamma'_i$.
\end{enumerate*}

\section{Transfinite words}
\label{sec:transf-sequ}

We call \emph{transfinite word} on a set $X$ any map from $\alpha$ to
$X$, where $\alpha$ is any ordinal.  Such a word $w$ has \emph{length}
$|w| \eqdef \alpha$.  Seen as the set of ordinals $\beta$ strictly
less than $\alpha$, $|w|$ is also the \emph{domain} of $w$, and we
write $\lat w \beta$ for the letter at position $\beta$ in $w$, for
every $\beta \in |w|$ (equivalently, $\beta < |w|$).

When $X$ is preordered by $\leq$,
the \emph{subword preordering} $\leq$ is defined on transfinite words
by $w \leq_* w'$ if and only if there is a strictly increasing map
$f \colon |w| \to |w'|$ such that
$\lat w \beta \leq \lat {w'} {f (\beta)}$ for every $\beta$,
$0\leq \beta < |w|$.
We also say that $f$ \emph{exhibits} $w$ as a subword of $w'$.

We write $X^{<\alpha}$ for the set of transfinite words $w$ of length
$|w| < \alpha$.  $X^{\leq\alpha}$ denotes $X^{<\alpha+1}$.  For
example, $X^{<\omega}$ is the set $X^*$ of finite words on $X$, and
$X^{<\omega+1} = X^{\leq\omega}$ is the set of finite-or-infinite
words studied in Part~I \cite{JGL:noethinf}.

The \emph{concatenation} $ww'$ of $w$ and $w'$ is the 
transfinite word of length $|w|+|w'|$ such that
$\lat {(ww')} \beta = \lat w \beta$ for every $\beta < |w|$, and
$\lat {(ww')} {|w|+\beta} = \lat {w'} \beta$ for every $\beta < |w'|$.
We write $AB$ for the set $\{ww' \mid w \in A, w' \in B\}$.

We are interested in the following topology.  We cannot work on the
(proper) class of all transfinite words over $X$, for foundational
reasons.  Instead we work on \emph{sets} $Y$ of transfinite words;
this leads us to take intersections with $Y$ here and there.  Usually,
$Y$ will be a set of the form $X^{<\alpha}$.
\begin{definition}[Regular subword topology]
  \label{defn:subword:topo}
  The \emph{regular subword topology} on any set of transfinite words
  $Y$ on a space $X$ is the coarsest topology that makes the sets
  $(F_1^{<\alpha_1} F_2^{<\alpha_2} \allowbreak \cdots \allowbreak
  F_n^{<\alpha_n}) \cap Y$ closed, where $n \in \nat$, $F_1$, $F_2$,
  \ldots, $F_n$ are closed subsets of $X$, and $\alpha_1$, $\alpha_2$,
  \ldots, $\alpha_n$ are ordinals.
\end{definition}

The following class of ordinals will be ubiquitous.
\begin{definition}[Bound]
  \label{defn:bound}
  A \emph{bound} is an ordinal of the form $\omega^\beta$ or
  $\omega^\beta+1$, $\beta \geq 0$.  The \emph{trivial bound} is
  $\omega^0$ ($=1$), all others are \emph{non-trivial}.  A
  \emph{proper bound} is one of the form $\omega^\beta$ or
  $\omega^\beta+1$ with $\beta \geq 1$.
\end{definition}

\begin{definition}[Preatom, atom, product]
  \label{defn:product}
  Let $X$ be a topological space.  A \emph{preatom} is
  an expression of the form $F^{<\gamma}$, where $F$ is a closed
  subset of $X$ and $\gamma$ is a bound.
  An \emph{atom} is a preatom $F^{< \gamma}$ such that $\gamma$ is
  non-trivial, and if $\gamma = \omega^0+1$ then $F$ is irreducible
  closed in $X$.

  A \emph{transfinite product} $P$ is any set of the form
  $A_1 A_2 \cdots A_n$, where $n \in \nat$ and each $A_i$ is an atom.
  We write $\varepsilon$ when $n=0$, namely
  $\varepsilon \eqdef \{\epsilon\}$.
\end{definition}
When $\beta=0$, $F^{<\omega^0}=\varepsilon = \{\epsilon\}$, and
$F^{<\omega^0+1}$ is sometimes written as $F^?$: that is the set of
words of length at most $1$, whose only letter if any is in $F$.

\begin{lemma}
  \label{lemma:F:split}
  For every set $F$, for every ordinal $\alpha$, one can write
  $F^{<\alpha}$ as
  $F^{<\gamma_1} F^{<\gamma_2} \allowbreak \cdots \allowbreak
  F^{<\gamma_m}$, where $m \in \nat$ and each $\gamma_i$ is a
  non-trivial bound.
\end{lemma}
\begin{proof}
  If $\alpha=0$, we take $m \eqdef 0$.  Otherwise, let us write
  $\alpha$ in Cantor normal form, as a finite sum
  $\gamma_1 + \cdots + \gamma_m$ of indecomposable ordinals, where
  $m \neq 0$ and $\alpha \geq \gamma_1 \geq \cdots \geq \gamma_m$.  We
  claim that $F^{< \alpha}$ is equal to
  $F^{\leq \gamma_1} F^{\leq \gamma_2} \cdots F^{\leq \gamma_{m-1}}
  F^{<\gamma_m}$.  The result will follow, since the latter is equal
  to
  $F^{< \gamma_1+1} F^{< \gamma_2+1} \cdots F^{< \gamma_{m-1}+1} F^{<
    \gamma_m}$: when $\gamma_m \neq 1$, all the superscripts are
  non-trivial bounds; when $\gamma_m = 1$, this simplifies to
  $F^{< \gamma_1+1} F^{< \gamma_2+1} \allowbreak \cdots \allowbreak
  F^{< \gamma_{m-1}+1}$, where all superscripts are non trivial
  bounds, too.

  $F^{\leq \gamma_1} 
  \cdots F^{\leq \gamma_{m-1}} F^{<\gamma_m}$ is included in
  $F^{< \alpha}$.  Conversely, let $w \in F^{< \alpha}$, and let us
  write $|w|$ in Cantor normal form as
  $\gamma'_1 + \cdots + \gamma'_n$.  We can write $w$ as a
  concatenation $w_1 \cdots w_n$ where $|w_1|=\gamma'_1$, \ldots,
  $|w_n|=\gamma'_n$. Since $|w| < \alpha$, there is a number
  $i \geq 1$ such that $\gamma'_1=\gamma_1$, \ldots,
  $\gamma'_{i-1}=\gamma_{i-1}$ and either $i=n+1 \leq m$, or
  $i \leq m, n$ and $\gamma'_i < \gamma_i$. In the first case,
  $w \in F^{\leq \gamma_1} \cdots F^{\leq \gamma_n} \subseteq F^{\leq
    \gamma_1} 
  \cdots F^{\leq \gamma_{m-1}} F^{<\gamma_m}$.  In the second case,
  since
  $\gamma_i > \gamma'_i \geq \gamma'_{i+1} \geq \cdots \geq \gamma'_n$
  and $\gamma_i$ is indecomposable,
  $\gamma'_i + \cdots + \gamma'_n < \gamma_i$.  Then
  $w_1 \in F^{\leq \gamma_1}$, \ldots,
  $w_{i-1} \in F^{\leq \gamma_{i-1}}$, and $w_i w_{i+1} \cdots w_n$ is
  in $F^{<\gamma_i}$, hence in
  $F^{\leq \gamma_i} \cdots F^{\leq \gamma_{m-1}} F^{< \gamma_m}$.
\end{proof}

\begin{proposition}
  \label{prop:subword:topo}
  The regular subword topology on any set of transfinite words $Y$ on
  a Noetherian space $X$ is the coarsest topology that has the
  intersections of $Y$ with transfinite products as closed sets.
\end{proposition}
\begin{proof}
  Let us consider a set of the form
  $F_1^{<\alpha_1} F_2^{<\alpha_2} \cdots F_n^{<\alpha_n}$, where
  $F_1$, \ldots, $F_n$ are closed subsets of $X$, and $\alpha_1$,
  $\alpha_2$, \ldots, $\alpha_n$ are ordinals.  We claim that we can
  rewrite it as a finite union of transfinite products.

  If some $\alpha_i$ equals $0$, then
  $F_1^{<\alpha_1} F_2^{<\alpha_2} \cdots F_n^{<\alpha_n}$ is empty.
  Otherwise, using Lemma~\ref{lemma:F:split}, we may assume that every
  $\alpha_i$ is a non-trivial bound.  We may also remove the preatoms
  $F_i^{<\alpha_i}$ such that $F_i = \emptyset$, since in that case
  $F_i^{\alpha_i} = \varepsilon$.  Let $I$ be the subset of those
  indices $i$, $1\leq i\leq n$, such that $\alpha_i = \omega^0+1$.
  For each $i \in I$, we can write $F_i$ as a finite union of
  irreducible closed subsets $C_{i1}$, \ldots, $C_{ik_i}$ (and
  $k_i \neq 0$ since $F_i \neq \emptyset$), since $X$ is Noetherian.
  For every function $f$ mapping each $i \in I$ to an element of
  $\{1, \cdots, k_i\}$, let $P_f$ be the transfinite product obtained
  from $F_1^{<\alpha_1} F_2^{<\alpha_2} \cdots F_n^{<\alpha_n}$ by
  replacing each preatom $F_i^{<\alpha_i}$, $i \in I$, by
  $C_{if (i)}^?$.  Then
  $F_1^{<\alpha_1} F_2^{<\alpha_2} \cdots F_n^{<\alpha_n}$ is the
  finite union of the transfinite products $P_f$, when $f$ varies over
  the finitely many possible functions.
\end{proof}


\section{Elementary combinatorics on transfinite words}
\label{sec:comb-transf-words}

\begin{lemma}
  \label{lemma:bounds}
  Let $\gamma$, $\gamma'$ be two bounds, and $u$ and $v$ be
  transfinite words.
  \begin{enumerate}
  \item If $\gamma$ is indecomposable and $|u|, |v| < \gamma$, then
    $|uv| < \gamma$.
  \item If $\gamma < \gamma'$ and $|u| < \gamma$, $|v| < \gamma'$,
    then $|uv| < \gamma'$.
  \item If $\gamma$ is proper and $|u|<\gamma$, then for every
    $x \in X$, $|xw| < \gamma$.
  \end{enumerate}
\end{lemma}
\begin{proof}
  (1) By definition of indecomposability, using $|uv| = |u| + |v|$.

  (2) follows from (1) if $\gamma'$ is indecomposable.  Otherwise, let
  us write $\gamma'$ as $\omega^\beta + 1$.  Then $\gamma < \gamma'$
  means $\gamma \leq \omega^\beta$.  Since $|u| < \gamma$,
  $|u| < \omega^\beta$ and since $|v| < \gamma'$,
  $|v| \leq \omega^\beta$.  Hence
  $|uv| = |u| + |v| \leq \omega^\beta < \gamma'$, because
  $\omega^\beta$ is indecomposable.

  (3) Since $\gamma$ is proper, we have $|x| = \omega^0 < \gamma$;
  then (3) follows from (2).
\end{proof}

\begin{lemma}
  \label{lemma:prod}
    Every closed set in the regular subword topology is
    downwards-closed with respect to $\leq_*$.
\end{lemma}
\begin{proof}
  Given any two downwards-closed subsets $A$ and $B$ with respect to
  $\leq_*$, their product $AB$ is, too: if $w \leq_* w'$ and
  $w' \in AB$, then we can write $w'$ as $u'v'$ where $u' \in A$ and
  $v' \in B$, and it is then easy to show that $w = uv$ for some
  $u \leq_* u'$ and $v \leq_* v'$.  It is clear that every atom is
  downwards-closed, as well as $\varepsilon$, so every transfinite
  product is downwards-closed, hence also every intersection of finite
  unions of transfinite products.
\end{proof}

For every indecomposable ordinal $\gamma$, there is a so-called
\emph{Hessenberg pairing} map
$H \colon \gamma \times \gamma \to \gamma$, which is injective; for
all $\alpha < \beta < \gamma$ and $\delta < \gamma$,
$H (\alpha, \delta) < H (\beta, \delta)$ and
$H (\delta, \alpha) < H (\delta, \beta)$
\cite[Exercise~2.23~(ii)]{Levy:basicset},
and it is easy to see that
$H (\alpha, 0) \geq \alpha$ for every $\alpha < \gamma$.

\begin{lemma}
  \label{lemma:stutter}
  Let $F$ be a non-empty subset of a set $X$.  For every transfinite
  word $w$ on $X$ such that $|w|$ is an indecomposable ordinal
  $\gamma$, and whose letters are in $F$, there is a word $w'$ of
  length $\gamma$, whose letters are in $F$ again, such that for every
  way of writing $w'$ as a concatenation $uv$ with $|u| < \gamma$,
  $w \leq_* v$.
\end{lemma}
\begin{proof}
  Let us pick $x \in F$.  We build $w'$ as the following word of
  length $\gamma$: $\lat {w'} {H (\alpha, \beta)} \eqdef \lat w \beta$
  for all $\alpha, \beta < \gamma$, and $\lat {w'} \delta \eqdef x$
  for every position $\delta < \gamma$ that is not in the range of
  $H$.  Now let us write $w'$ as $uv$ with $|u| < \gamma$.  Then
  $H (|u|, \_)$ exhibits $w$ as a subword of $v$, using the fact that
  $H (|u|, 0) \geq |u|$.
\end{proof}

\section{Continuity and irredundancy}
\label{sec:cont-prop}

\begin{lemma}
  \label{lemma:cat:cont}
  Let $X$ be a Noetherian space.
  \begin{enumerate}
  \item For any set $Y$ of transfinite words on $X$ containing
    $X^{\leq 1}$, the function $i \colon X \to Y$ mapping $x \in X$ to
    the one-letter word $x$ is continuous.
  \item For every ordinal $\beta$, the concatenation map
    $cat \colon X^{<\omega^\beta} \times X^{\leq \omega^\beta} \to
    X^{\leq \omega^\beta}$ (resp.,
    $cat \colon X^{<\omega^\beta} \times X^{< \omega^\beta} \to X^{<
      \omega^\beta}$) is continuous.
  \end{enumerate}
\end{lemma}
\begin{proof}
  (1) For every preatom $F^{<\gamma}$ with $\gamma$ non-trivial,
  $i^{-1} (F^{<\gamma}) = F$ is closed.  Then, the inverse image of
  any transfinite product $A_1 A_2 \cdots A_n \cap Y$ is
  $i^{-1} (A_1) \cup \cdots \cup i^{-1} (A_n)$, which is closed.

  (2) First, $cat$ is well-defined by Lemma~\ref{lemma:bounds}~(1).
  As far as continuity is concerned, let $P \eqdef A_1 A_2 \cdots A_n$
  be any transfinite product.  We show that $cat^{-1} (P)$ (or rather,
  $cat^{-1} (P \cap X^{\leq \omega^\beta} )$) is (the intersection of
  $X^{<\omega^\beta} \times X^{\leq \omega^\beta}$ with) a finite
  union $\bigcup_{i=1}^m P_i \times Q_i$ of products of pairs of
  transfinite word products $P_i$, $Q_i$, by induction on $n$.  If
  $n=0$, then $cat^{-1} (P)$ only contains $(\epsilon, \epsilon)$,
  hence is equal to $\varepsilon \times \varepsilon$.  Otherwise, let
  $P' \eqdef A_2 \cdots A_n$, and let us write $A_1$ as $F^{<\gamma}$,
  where $\gamma$ is a non-trivial bound, and $F$ is closed.
  By induction hypothesis, $cat^{-1} (P')$ is a finite union
  $\bigcup_{i=1}^m P_i \times Q_i$, where $P_i$ and $Q_i$ are
  transfinite word products.  The pairs of words $u$, $v$ whose
  concatenation are in $P$ are those such that $u$ is of the form
  $u_1 u_2$ with $u_1 \in A_1$ and $(u_2, v) \in cat^{-1} (P')$
  (namely, the elements of $\bigcup_{i=1}^m A_1 P_i \times Q_i$,
  since, as one sees easily, concatenation distributes over union) or
  such that $v$ is of the form $v_1 v_2$ with $u v_1 \in A_1$ and
  $v_2 \in P'$.  In order to conclude, it therefore suffices to show
  that the set $A$ of pairs $(u, v)$ with $v$ of the form $v_1 v_2$,
  $u v_1 \in A_1$ and $v_2 \in P'$, is a finite union of products of
  pairs of transfinite products.

  If $\gamma$ is of the form $\omega^\beta$, then $u v_1 \in A_1$ if
  and only if $u \in F^{<\gamma}$ and $v_1 \in F^{<\gamma}$.  The only
  if direction is clear, and the if direction is by
  Lemma~\ref{lemma:bounds}~(1).  Hence $A = F^{<\gamma} \times
  F^{<\gamma} P'$ in this case.
  
  If $\gamma$ is of the form $\omega^\beta+1$, then $u v_1 \in A_1$ if
  and only if $u \in F^{<\omega^\beta}$ and
  $v_1 \in F^{\leq\omega^\beta}$, or $u \in F^{\leq\omega^\beta}$ and
  $v_1=\epsilon$.  In the only if direction, we reason by cases,
  depending whether $|u| = \omega^\beta$ or not.  In the if direction,
  the case $v_1 = \epsilon$ is obvious, while
  $u \in F^{<\omega^\beta}$ and $v_1 \in F^{\leq\omega^\beta}$ imply
  $u v_1 \in A_1 = F^{<\gamma}$ by Lemma~\ref{lemma:bounds}~(2).
  Hence
  $A = (F^{<\omega^\beta} \times F^{<\gamma} P') \cup (F^{<\gamma}
  \times P')$ in this case.
\end{proof}

On spaces of the form $X^{<\alpha}$, 
the following refinement of the notion of transfinite product will be
the family $\mathcal P$ we will use Proposition~\ref{prop:S} on.
\begin{definition}[$\alpha$-product]
  \label{defn:alpha:prod}
  For a topological space $X$ and a bound $\alpha$, the
  \emph{$\alpha$-products} are the products of the form
  ${F_1}^{<\gamma_1} {F_2}^{<\gamma_2} \cdots {F_n}^{<\gamma_n}$ where
  $n \in \nat$,
  \begin{itemize}
  \item $\gamma_i \leq \alpha$ and $F_i$ is non-empty for each $i$,
    $1\leq i\leq n$,
  \item and if $\alpha$ is decomposable, then
    $\gamma_i < \alpha$ for every $i$, $1\leq i < n$; namely, the only
    $\gamma_i$ that is equal to $\alpha$, if any, is obtained with
    $i=n$.
  \end{itemize}
\end{definition}

\begin{proposition}
  \label{prop:alpha:prod}
  For every topological space $X$ and every bound $\alpha$, the
  $\alpha$-products are closed in $X^{<\alpha}$, and their complements
  form a subbase of the regular subword topology.
\end{proposition}
\begin{proof}
  Let
  $P \eqdef {F_1}^{<\gamma_1} {F_2}^{<\gamma_2} \cdots
  {F_n}^{<\gamma_n}$ be an $\alpha$-product.  Let us consider any
  transfinite word $w$ in $P$, and let us write it as
  $u_1 u_2 \cdots u_n$, where $u_i \in {F_i}^{<\gamma_i}$ for each
  $i$.  If $\alpha$ is indecomposable, then $|w| < \alpha$ because
  $|u_i| < \alpha$ for every $i$, and using
  Lemma~\ref{lemma:bounds}~(1).  Otherwise, let us write $\alpha$ as
  $\omega^\beta+1$.  If $\gamma_i \leq \omega^\beta$ for every $i$, by
  the same argument $|w| < \omega^\beta$, hence $|w| < \alpha$.  By
  the second item in the definition of $\alpha$-products, the only
  remaining possibility is that $\gamma_i \leq \omega^\beta$ for every
  $i$, $1\leq i< n$, that $n \geq 1$ and that
  $\gamma_n = \omega^\beta+1 = \alpha$.  Then $|u_n| < \alpha$; since
  $|u_{n-1}| < \gamma_{n-1} < \alpha$, we obtain
  $|u_{n-1} u_n| < \alpha$ by Lemma~\ref{lemma:bounds}~(2); then
  $|u_{n-2} u_{n-1} u_n| < \alpha$, \ldots, and eventually
  $|w| < \alpha$.

  To show the second part of the proposition, we claim that the
  intersection of every transfinite product
  $P \eqdef {F_1}^{\gamma_1} {F_2}^{<\gamma_2} \cdots
  {F_n}^{<\gamma_n}$ with $X^{<\alpha}$ is a finite union of
  $\alpha$-products.  If $n=0$, $P$ is already an $\alpha$-product, so
  let us assume $n \neq 0$.  Given any $w \in P \cap X^{<\alpha}$, we
  can write $w$ as $u_1 u_2 \cdots u_n$, where
  $u_i \in {F_i}^{<\gamma_i}$ for each $i$.  For each $i$, not only
  $|u_i| < \gamma_i$ but also $|u_i| \leq |w| < \alpha$, so we can
  assume without loss of generality that $\gamma_i \leq \alpha$ for
  every $i$, $1\leq i\leq n$.

  If $\alpha$ is indecomposable, then this makes $P$ an
  $\alpha$-product.  Henceforth let us assume that $\alpha$ is
  decomposable, say $\alpha \eqdef \omega^\beta+1$.  Then we can
  rewrite every transfinite product of the form $F^{<\alpha} Q$
  included in $X^{<\alpha}$ as $F^{<\omega^\beta} Q \cup F^{<\alpha}$.
  Indeed, every word $w \eqdef uv$ with $u \in F^{<\alpha}$ and
  $v \in Q$ is either such that $|u| < \omega^\beta$ (then
  $w \in F^{<\omega^\beta} Q$), or $|u| = \omega^\beta$.  If
  $|u|=\omega^\beta$, then $v=\epsilon$, otherwise
  $|w| > \omega^\beta$, which is impossible since
  $w \in X^{< \alpha}$; so $w \in F^{<\alpha}$.

  We can therefore rewrite $P$ as follows.  Let $i_1 < \cdots < i_k$
  be the list of indices $i$ between $1$ and $n-1$ such that
  $\gamma_i = \alpha$.  Let $\gamma'_i \eqdef \gamma_i$ for every $i$
  different from $i_1$, \ldots, $i_k$, and
  $\gamma'_i \eqdef \omega^\beta$ otherwise.  (I.e., we replace the
  exponents equal to $\alpha = \omega^\beta+1$ by $\omega^\beta$.)
  Then $P$ is the union of the $\alpha$-products
  ${F_1}^{<\gamma'_1} {F_2}^{<\gamma'_2} \cdots
  {F_{i_j-1}}^{<\gamma'_{i_j-1}} {F_{i_j}}^{<\alpha}$,
  $1 \leq j \leq k$, and
  ${F_1}^{<\gamma'_1} {F_2}^{<\gamma'_2} \cdots {F_n}^{<\gamma'_n}$.
\end{proof}

\begin{proposition}
  \label{prop:irred}
  Let $X$ be a Noetherian space.  For every bound $\alpha$, every
  $\alpha$-product is irreducible in $X^{<\alpha}$.  Hence
  the family of $\alpha$-products is irredundant.
\end{proposition}
\begin{proof}
  We first show that every atom $F^{<\gamma}$ is irreducible in
  $X^{<\alpha}$.

  If $\gamma = \omega^0+1$ and $F$ is irreducible, then $i$ is
  continuous (Lemma~\ref{lemma:cat:cont}~(1)), so
  $\Sober i (F) = cl (i [F])$ is irreducible closed in
  $X^{<\alpha}$. Then $cl (i [F])$ is non-empty, and downwards-closed
  under $\leq_*$ by Lemma~\ref{lemma:prod}, so it contains
  $\epsilon$.  Clearly $i [F] \subseteq cl (i [F])$, so
  $F^{<\gamma} = F^? = \{\epsilon\} \cup i [F] \subseteq cl (i [F])$.
  Also, $i [F] \subseteq F^{<\gamma}$, and $F^{<\gamma}$ is closed.
  Therefore $F^{<\gamma} = cl (i [F])$, so $F^{<\gamma}$ is
  irreducible.

  Let now $\gamma$ be a proper bound.  We show that $F^{<\gamma}$ is
  directed in $\leq_*$, namely that $F^{<\gamma}\neq \emptyset$ and
  that any two elements $u$, $v$ of $F^{<\gamma}$ have an upper bound
  $w$ in $F^{<\gamma}$.  If $\gamma$ is indecomposable, then
  $w \eqdef uv$ fits, using Lemma~\ref{lemma:bounds}~(1).  Otherwise,
  let $\gamma \eqdef \omega^\beta+1$.  If $|u| < \omega^\beta$, then
  $uv$ fits again, by Lemma~\ref{lemma:bounds}~(2).  If
  $|v| < \omega^\beta$, then we pick $vu$ instead.  Finally, if
  $|u| = |v| = \omega^\beta$, we define $w$ as the one-for-one
  interleaving of $u$ and $v$, namely as the word of length
  $\omega^\beta$ such that, for every ordinal
  $\lambda + n < \omega^\beta$ (where $\lambda$ is $0$ or a limit
  ordinal, and $n \in \nat$),
  $\lat w {\lambda + 2n} = \lat u {\lambda+n}$ and
  $\lat w {\lambda + 2n+1} = \lat v {\lambda + n}$.

  Since $F^{<\gamma}$ is directed, $F^{<\gamma}$ is irreducible.
  Indeed, if $F^{<\gamma} \subseteq \mathcal C_1 \cup \mathcal C_2$
  where $\mathcal C_1$ and $\mathcal C_2$ are closed in $X^{<\alpha}$,
  but $F^{<\gamma} \not\subseteq \mathcal C_1, \mathcal C_2$, then we
  can pick $u \in F^{<\gamma} \diff \mathcal C_1$ and
  $v \in F^{<\gamma} \diff \mathcal C_2$; let $w$ be an upper bound of
  $u$ and $v$ in $F^{<\gamma}$, then $w$ is neither in $\mathcal C_1$
  nor in $\mathcal C_2$, since those sets are downwards-closed under
  $\leq_*$ (Lemma~\ref{lemma:prod}), and therefore $w \in F^{<\gamma}
  \diff (\mathcal C_1 \cup \mathcal C_2)$, which is impossible.

  We now prove that every $\alpha$-product $P$ is irreducible in
  $X^{<\alpha}$, by induction on the number $n$ of atoms in $P$.  If
  $n=0$, then $P = \varepsilon$, and this is clear.  If $n\geq 1$, we
  have just seen that $P = F^{<\gamma}$ is irreducible.  If
  $n \geq 2$, we write $P$ as $F^{<\gamma} Q$ where $Q$ is shorter,
  hence irreducible by induction hypothesis.
  
  If $\alpha$ is indecomposable, then 
  $F^{<\gamma}$ is irreducible in $X^{<\alpha}$, so
  $F^{<\gamma} \times Q$ is irreducible in
  $X^{<\alpha} \times X^{<\alpha}$.  By
  Lemma~\ref{lemma:cat:cont}~(2), $cat$ is continuous from the latter
  to $X^{<\alpha}$; so
  $\Sober {cat} (F^{<\gamma} \times Q) = cl (cat [F^{<\gamma} \times
  Q])$ is irreducible in $X^{<\alpha}$.  Now
  $cat [F^{<\gamma} \times Q] = F^{<\gamma} Q$ is closed, hence equal
  to its own closure.

  If $\alpha$ is decomposable, then let
  $\alpha \eqdef \omega^\beta+1$.  By the second item in the
  definition of $\alpha$-products, we must have $\gamma < \alpha$,
  hence $\gamma \leq \omega^\beta$.  Then $F^{<\gamma}$ is irreducible
  in $X^{<\omega^\beta}$, while $Q$ is irreducible in
  $X^{\leq\omega^\beta}$.  We use Lemma~\ref{lemma:cat:cont}~(2):
  $cat$ is continuous from
  $X^{<\omega^\beta} \times X^{\leq\omega^\beta}$ to
  $X^{\leq\omega^\beta}$.  Then we conclude as above that
  $F^{<\gamma} Q = cat [F^{<\gamma} \times Q] = cl (cat [F^{<\gamma}
  \times Q])$ is irreducible.
\end{proof}

\section{Inclusion of transfinite products}
\label{sec:incl-transf-prod}

We start with necessary conditions for inclusion of transfinite
products. We abbreviate ``$\gamma= \gamma'$ and $\gamma$ is
decomposable'' as ``$\gamma=\gamma'$ is decomposable''.

\begin{lemma}
  \label{lemma:incl:=>}
  Let $\gamma$, $\gamma'$ be two non-trivial bounds, and let
  $F^{<\gamma} P$, ${F'}^{<\gamma'} P'$ be two transfinite products.
  If $F^{<\gamma} P \subseteq {F'}^{<\gamma'} P'$, then:
  \begin{enumerate}
  \item $P \subseteq  {F'}^{<\gamma'} P'$.
  \item If $F \not\subseteq F'$ then $F^{<\gamma} P \subseteq P'$.
  \item If $F$ is non-empty and if $\gamma > \gamma'$, then
    $F^{<\gamma} P \subseteq P'$.
  \item If $F \subseteq F'$ and $F \neq \emptyset$, and if
    $\gamma = \gamma'$ is decomposable, then $P \subseteq P'$.
  \end{enumerate}
\end{lemma}
\begin{proof}
  Let us assume that $F^{<\gamma} P \subseteq {F'}^{<\gamma'} P'$.
  
  (1) $P \subseteq F^{<\gamma} P$, because $\epsilon \in F^{<\gamma}$,
  and
  $F^{<\gamma} P \subseteq {F'}^{<\gamma'} P'$ by assumption.
 
  (2) Let us assume $F \not\subseteq F'$.  Then there is a letter $x$
  in $F \diff F'$. Let $w \in F^{<\gamma} P$ be arbitrary,
  and let us
  write $w$ as $uv$ where $u \in F^{<\gamma}$ and $v \in P$.

  If $\gamma$ is proper, then $xu$ is in $F^{<\gamma}$ again, by
  Lemma~\ref{lemma:bounds}~(3), so $xw = xuv$ is in
  $F^{<\gamma} P$, hence in ${F'}^{<\gamma'} P'$. Since
  $x \not\in F'$, $xuv$ is in $P'$, so $w = uv$ is
  in $P'$, by
  Lemma~\ref{lemma:prod} and since $w \leq_* xuv$.
  Since $w$ is arbitrary,
  $F^{<\gamma} P \subseteq P'$.

  It remains to deal with the case $\gamma = \omega^0+1$, $F$
  irreducible. Let
  $A \eqdef \{y \in X \mid \forall v \in P, yv \in P'\}$. For every $y
  \in F \diff F'$, $y$ is in $A$: indeed, for every
  $v \in P$, $yv$ is in $F^? P = F^{<\gamma} P$ hence in
  ${F'}^{<\gamma'} P'$, and since $y \not\in F'$, $yv$ must be in
  $P'$. This means that $F \subseteq F' \cup A$. $A$ is also
  equal to $\bigcap_{v \in P} cat (i (\_), v)^{-1} (P')$, where
  $cat (i (\_), v) \colon y \mapsto cat (i (y), v) = yv$ is continuous
  by Lemma~\ref{lemma:cat:cont}. Hence $A$ is closed.
  Since $F$ is irreducible, and since
  $F \not\subseteq F'$, $F$ must be included in $A$; equivalently,
  $FP \subseteq P'$.  Since $P'$ is downwards-closed with respect to
  $\leq_*$ (Lemma~\ref{lemma:prod}), and since every element
  $u$ of $P$ is a subword of $xu \in FP$, $F^? P = P \cup FP$ is also
  included in $P'$---namely, $F^{<\gamma} P \subseteq P'$.

  (3) Let us assume $F^{<\gamma} P \subseteq {F'}^{<\gamma'} P'$, with
  $F \neq \emptyset$ and $\gamma > {\gamma'}$. We pick $x \in F$.
  For every $\alpha$, let $x^\alpha$ be the word
  of length $\alpha$ whose sole letter is $x$.

  Whether $\gamma'$ is equal to $\omega^{\beta'}$ or to
  $\omega^{\beta'+1}$, we define ${\gamma'}^-$ as $\omega^{\beta'}$.
  For
  every transfinite word $v$, if
  $x^{{\gamma'}^-} v \in {F'}^{<\gamma'} P'$, then $v \in P'$.
  Indeed, assuming 
  $x^{{\gamma'}^-} v \in {F'}^{<\gamma'} P'$, we can write
  $x^{{\gamma'}^-} v$ as $v_1 v_2$ where $v_1 \in {F'}^{<\gamma'}$ and
  $v_2 \in P'$.  If $\gamma' = \omega^{\beta'}$, then
  $|x^{{\gamma'}^-}| = \omega^{\beta'} > |v_1|$; if
  $\gamma' = \omega^{\beta'+1}$, then
  $|x^{{\gamma'}^-}| = \omega^{\beta'} \geq |v_1|$.  In any case,
  $|v_1| \leq |x^{{\gamma'}^-}|$, so $v_1$ is a prefix
  of $x^{{\gamma'}^-}$, and $v_2$ is of the form $x^\alpha v$
  for some ordinal $\alpha$.
  Since $v_2 \in P'$, and $v \leq_* x^\alpha v = v_2$, 
  $v \in P'$ by Lemma~\ref{lemma:prod}.

  Let $w \in F^{<\gamma} P$ be arbitrary, and let us write $w$ as $uv$ where
  $u \in F^{<\gamma}$ and $v \in P$.  With the aim of showing that $w
  \in P'$, we form the transfinite word
  $x^{{\gamma'}^-} u$.  Its letters are in $F$, and its length is
  ${\gamma'}^- + |u|$.  

  If ${\gamma'}^- + |u| < \gamma$, then $x^{{\gamma'}^-} u$ is in
  $F^{<\gamma}$, so $x^{{\gamma'}^-} w = x^{{\gamma'}^-} u v$ is in
  $F^{<\gamma} P$, hence in ${F'}^{<\gamma'} P'$.  We have seen that
  this implies $w \in P'$.

  Henceforth, we assume that ${\gamma'}^- + |u| \geq \gamma$. We recall
  that ${\gamma'}^- \leq \gamma' < \gamma$ and that $|u| < \gamma$.
  If $\gamma$ were indecomposable, then we would have
  ${\gamma'}^- + |u| < \gamma$, contradicting our
  assumption. Hence $\gamma = \omega^\beta+1$ for
  some ordinal $\beta$.  Then $\gamma' \leq \omega^\beta$ and
  $|u| \leq \omega^\beta$.  If $\gamma' < \omega^\beta$, then
  ${\gamma'}^- + |u| \leq \omega^\beta < \gamma$, which is impossible
  again.  Therefore $\gamma' = \omega^\beta$.

  To sum up, $|u| < \gamma = \omega^\beta+1$, and
  $\gamma' = \omega^\beta$.  Let $W$ be $u$ itself if
  $|u| = \omega^\beta$, else $ux^{\omega^\beta}$.  In each case, $W$
  is a word of length $\omega^\beta$ whose letters are all in $F$, so
  $W \in F^{<\gamma}$. We use
  Lemma~\ref{lemma:stutter}: let $W'$ be a word of length
  $\omega^\beta$, whose letters are all in $F$, and such that for
  every way of writing $W'$ as $UV$ with
  $|U| < \omega^\beta$, $W \leq_* V$.  Then $W'v$ is in
  $F^{<\gamma} P$, hence in $F'^{<\gamma'} P'$. Let us write $W'v$ as
  $Uv'$ where $|U| < \gamma'$ and $v' \in P'$. Since
  $|U| < \gamma' = \omega^\beta$, $U$ is a prefix of $W'$, and we can
  therefore write $W'$ as $UV$ for some transfinite word $V$, and $v'$
  as $V v$.  By construction, $W \leq_* V$.  Therefore
  $u \leq_* W \leq_* V$, and hence $w = uv \leq_* Vv = v'$.  Since
  $v'$ is in $P'$, so is $w$, by Lemma~\ref{lemma:prod}.


  (4) Since $F \neq \emptyset$, let us pick $x \in F$. We
  write $\gamma = \gamma'$ as $\omega^\beta + 1$.  For every
  $w \in P$, $x^{\omega^\beta} w$ is in $F^{<\gamma} P$, hence in
  ${F'}^{<\gamma'} P'$.  Hence we can write $x^{\omega^\beta} w$
  as $uv$ where $|u| < \gamma'$ (namely, $|u| \leq \omega^\beta$) and
  $v \in P'$.  Since $|u| \leq \omega^\beta = |x^{\omega^\beta}|$, we
  can write $v$ as $x^\alpha w$ for some ordinal $\alpha$.  In
  particular, $w \leq_* v$, and since $v \in P'$, $w$ is in $P'$, by
  Lemma~\ref{lemma:prod}.
\end{proof}

We turn to sufficient conditions. There are three cases, depending on
the relative positions and indecomposability statuses of $\gamma$ and
$\gamma'$.

\begin{lemma}
  \label{lemma:incl:<}
  Let $\gamma$, $\gamma'$ be two non-trivial bounds, and let
  $F^{<\gamma} P$, ${F'}^{<\gamma'} P'$ be two transfinite products.
  Assuming that $\gamma < \gamma'$, or that $\gamma = \gamma'$ is
  indecomposable, $F^{<\gamma} P \subseteq {F'}^{<\gamma'} P'$ if and
  only if:
  \begin{enumerate}
  \item $F \subseteq F'$ and $P \subseteq  {F'}^{<\gamma'} P'$,
  \item or $F \not\subseteq F'$ and $F^{<\gamma} P \subseteq P'$.
  \end{enumerate}
\end{lemma}
\begin{proof}
  The `only if' direction is by Lemma~\ref{lemma:incl:=>}~(1) and
  (2). We deal with the `if' direction. Note that
  $\gamma \leq \gamma'$; also, if $\gamma'$ is
  decomposable, then $\gamma < \gamma'$.
  
  (1) 
  For every $w \in F^{<\gamma} P$, let us write $w$ as $uv$ with
  $u \in F^{<\gamma}$ and $v \in P$.  Since
  $P \subseteq {F'}^{<\gamma'} P'$, $v$ is in ${F'}^{<\gamma'} P'$.
  Let us write $v$ as $v_1 v_2$ with $v_1 \in {F'}^{<\gamma'}$ and
  $v_2 \in P'$.  Then the letters of $uv_1$ are all in $F'$, and
  $|uv_1| < \gamma'$ by Lemma~\ref{lemma:bounds}~(2).  It follows
  that $uv_1$ is in ${F'}^{<\gamma'}$, so $w = uv_1v_2$ is in
  ${F'}^{<\gamma'} P'$.

  (2) If $F^{<\gamma} P \subseteq P'$,
  then
  $F^{<\gamma} P \subseteq P' \subseteq {F'}^{<\gamma'} P'$, where the
  last inequality is because every $w \in P'$ can be written as
  $\epsilon w \in {F'}^{<\gamma'} P'$.
\end{proof}

\begin{lemma}
  \label{lemma:incl:>}
  Let $\gamma$, $\gamma'$ be two non-trivial bounds, and
  $F^{<\gamma} P$, ${F'}^{<\gamma'} P'$ be two transfinite products.
  Assuming that $\gamma > \gamma'$,
  $F^{<\gamma} P \subseteq {F'}^{<\gamma'} P'$ if and only if:
  \begin{enumerate}
  \item $F$ is empty and $P \subseteq {F'}^{<\gamma'} P'$,
  \item or $F$ is non-empty and $F^{<\gamma} P \subseteq P'$.
  \end{enumerate}
\end{lemma}
\begin{proof}
  (1) If $F=\emptyset$ then $F^{<\gamma} P = P$, and the equivalence
  is clear.
  
  (2) 
  If $F^{<\gamma} P \subseteq P'$, then
  $F^{<\gamma} P \subseteq {F'}^{<\gamma'} P'$ since $P'$ is trivially
  included in ${F'}^{<\gamma'} P'$.  The `only if' direction is by
  Lemma~\ref{lemma:incl:=>}~(3).
\end{proof}

\begin{lemma}
  \label{lemma:incl:=}
  Let $\gamma$, $\gamma'$ be two non-trivial bounds, and let
  $F^{<\gamma} P$, ${F'}^{<\gamma'} P'$ be two transfinite products.
  Assuming that $\gamma = \gamma'$ is decomposable,
  $F^{<\gamma} P \subseteq {F'}^{<\gamma'} P'$ if and only if:
  \begin{enumerate}
  \item $F$ is empty and $P \subseteq {F'}^{<\gamma'} P'$,
  \item or $F$ is non-empty, $F \subseteq F'$, and $P \subseteq P'$,
  \item or $F \not\subseteq F'$ and $F^{<\gamma} P \subseteq P'$.
  \end{enumerate}
\end{lemma}
\begin{proof}
  If $F$ is empty then $F^{<\gamma} P = P$, so $F^{<\gamma} P
  \subseteq {F'}^{<\gamma'} P'$ is equivalent to $P \subseteq
  {F'}^{<\gamma'} P'$.  Henceforth, we assume $F$ non-empty.

  `If' direction.  If $F \subseteq F'$ and $P \subseteq P'$, then
  $F^{<\gamma} P \subseteq {F'}^{<\gamma'} P'$ is obvious (recall that
  $\gamma=\gamma'$).  If $F^{\gamma} P \subseteq P'$, then
  $F^{\gamma} P \subseteq P' \subseteq {F'}^{<\gamma'} P'$.

  `Only if'. Let us assume
  $F^{<\gamma} P \subseteq {F'}^{<\gamma'} P'$.  If $F \subseteq F'$,
  then $P \subseteq P'$ by Lemma~\ref{lemma:incl:=>}~(4).
  Otherwise, $F^{<\gamma} P \subseteq P'$ by Lemma~\ref{lemma:incl:=>}~(2).
\end{proof}

\section{Reduced products}
\label{sec:normal-forms}

We can write transfinite products in many equivalent ways. For
example, $\emptyset^{<\gamma} = \varepsilon$ for every non-trivial
bound $\gamma$. Here are a few other cases.
\begin{lemma}
  \label{lemma:prod:equal}
  Let $F$, $F'$ be non-empty closed subsets of a topological space
  $X$, and $\gamma$, $\gamma'$ be non-trivial bounds.  If $F \subseteq
  F'$, then:
  \begin{enumerate}
  \item If $\gamma < \gamma'$ or if $\gamma=\gamma'$ is
    indecomposable, 
    then
    $F^{<\gamma} {F'}^{<\gamma'} = F'^{<\gamma'}$.
  \item If $\gamma'$ is indecomposable, and $\gamma \leq \gamma'$, then
    ${F'}^{<\gamma'} F^{<\gamma} = {F'}^{<\gamma'}$.
  \end{enumerate}
\end{lemma}
\begin{proof}
  The right-hand sides are always included in the left-hand sides.

  (1) By Lemma~\ref{lemma:incl:<}, and remembering that
  $F \subseteq F'$,
  $F^{<\gamma} {F'}^{<\gamma'} \subseteq F'^{<\gamma'}$ if and only if
  ${F'}^{<\gamma'} \subseteq F'^{<\gamma'}$, which is simply true.


  (2) Since $\gamma'$ is indecomposable, by the same lemma,
  ${F'}^{<\gamma'} F^{<\gamma} \subseteq {F'}^{<\gamma'}$ if and only if
  $F^{<\gamma} \subseteq {F'}^{<\gamma'}$, which holds since
  $F \subseteq F'$ and $\gamma \leq \gamma'$.
\end{proof}
We will see that this leads to canonical forms for transfinite
products.  As in \cite[Theorem~4.22]{GLHKKS:ideals}, and to reduce
excessive pedantry related to the difference between syntax and
semantics, we write $\bA$, $\bB$ (resp., $\bP$, $\bQ$) to denote
atoms, resp.\ sequences of atoms (syntax), and $A$, $B$, $P$, $Q$ for
their respective semantics.  Hence if $\bP = \bA_1 \bA_2 \cdots \bA_n$
(as a sequence), then $P = A_1 A_2 \cdots A_n$ (as a product).  The
(syntactic) atoms $\bA$ are pairs $(F, \gamma)$ of a closed set $F$
and a non-trivial bound $\gamma$ (with $F$ irreducible if
$\gamma=\omega^0+1$), and then $A = F^{<\gamma}$.  Note that
$\bP = \bQ$ implies $P=Q$, but the converse may fail.

\begin{definition}[Reduced]
  \label{defn:reduced}
  A sequence of atoms $\bP \eqdef \bA_1,\bA_2 \cdots \bA_n$ on $X$,
  where $\bA_i \eqdef (F_i, \gamma_i)$ for each $i$, is \emph{reduced}
  if and only if:
  \begin{enumerate}
  \item $F_i$ is a non-empty closed subset of $X$ ($1\leq i\leq n$);
  \item for every $i$, $1\leq i <n$, such that $\gamma_i <
    \gamma_{i+1}$, $F_i$ is not included in $F_{i+1}$;
  \item for every $i$, $1\leq i <n$, such that $\gamma_i=\gamma_{i+1}$
    is indecomposable, $F_i$ and $F_{i+1}$ are incomparable;
  \item and for every $i$, $1\leq i <n$, such that $\gamma_i$ is
    indecomposable and $\gamma_i > \gamma_{i+1}$, $F_i$ does not
    contain $F_{i+1}$.
  \end{enumerate}
\end{definition}

\begin{lemma}
  \label{lemma:red:1}
  For all non-empty closed subsets $F$ and $F'$ of a space $X$, for
  all non-trivial bounds $\gamma$, $\gamma'$,
  \begin{enumerate}
  \item $F^{<\gamma} \neq \varepsilon$;
  \item $F^{<\gamma} \subseteq {F'}^{<\gamma'}$ if and only if $F
    \subseteq F'$ and $\gamma \leq \gamma'$.
  \item $F^{<\gamma} = {F'}^{<\gamma'}$ if and only if $F=F'$ and
    $\gamma=\gamma'$.
  \end{enumerate}
\end{lemma}
\begin{proof}
  (1) Since $F$ is non-empty, let us pick $x$ in $F$.  Since $\gamma$
  is non-trivial, the one-letter word $x$ is in $F^{<\gamma}$, whence
  the conclusion.
  
  (2) If $\gamma < \gamma'$ or if $\gamma = \gamma'$ is
  indecomposable, then by Lemma~\ref{lemma:incl:<},
  $F^{<\gamma} \subseteq {F'}^{<\gamma'}$ if and only if
  $F \subseteq F'$ and $\varepsilon \subseteq {F'}^{<\gamma'}$ (true),
  or $F \not\subseteq F'$ and
  $F^{<\gamma} \subseteq \varepsilon$ (false, by (1)).
  Hence $F^{<\gamma} \subseteq {F'}^{<\gamma'}$ if and only if
  $F \subseteq F'$ in this case.
  If $\gamma > \gamma'$, by Lemma~\ref{lemma:incl:>}
  $F^{<\gamma} \subseteq {F'}^{<\gamma'}$ reduces to
  $F^{<\gamma} \subseteq \varepsilon$, which is false by (1).
  If $\gamma = \gamma'$ is decomposable, then by
  Lemma~\ref{lemma:incl:=}, $F^{<\gamma} \subseteq {F'}^{<\gamma'}$ if
  and only if $F \subseteq F'$ and
  $\varepsilon \subseteq \varepsilon$, or $F \not\subseteq F'$ and
  $F^{<\gamma} \subseteq \varepsilon$; equivalently, if
  $F \subseteq F'$.

  (3) follows immediately from (2).
\end{proof}

\begin{lemma}
  \label{lemma:red:eps}
  The only reduced sequence of atoms with semantics $\{\epsilon\}$
  is $\varepsilon$.
\end{lemma}
\begin{proof}
  Let $\bP$ be a reduced sequence of atoms of length $n \geq 1$, say
  $\bA_1 \cdots \bA_n$.  By Lemma~\ref{lemma:red:1}~(1), each $A_i$
  contains a non-empty word, and their concatenation is a non-empty
  word in $P$.
\end{proof}

\begin{lemma}
  \label{lemma:red:2}
  For all atoms $A \eqdef F^{<\gamma}$, $B \eqdef {F'}^{<\gamma'}$,
  with $F, F' \neq \emptyset$, and all transfinite products $P$, $Q$, if
  $AP \subseteq BQ$ and $A \not\subseteq B$, then $AP \subseteq Q$.
\end{lemma}
\begin{proof}
  Since $A \not\subseteq B$, we have $F \not\subseteq F'$ or
  $\gamma > \gamma'$, by Lemma~\ref{lemma:red:1}~(2).

  If $\gamma > \gamma'$, then Lemma~\ref{lemma:incl:>} and
  $F^{\gamma} P \subseteq {F'}^{<\gamma'} Q$ imply
  $F^{<\gamma} P \subseteq Q$.  Let us therefore assume
  $\gamma \leq \gamma'$, and $F \not\subseteq F'$.  If
  $\gamma < \gamma'$ or $\gamma=\gamma'$ is indecomposable, then by
  Lemma~\ref{lemma:incl:<}, $F^{<\gamma} P \subseteq Q$ again.  If
  $\gamma=\gamma'$ is decomposable, then we reach the same
  conclusion by using Lemma~\ref{lemma:incl:=}.
\end{proof}

\begin{lemma}
  \label{lemma:red:3a}
  For any reduced sequence of atoms $\bA \bP$, $AP=A$ implies
  $\bP = \varepsilon$.
\end{lemma}
\begin{proof}
  By contradiction, let us assume that $\bP = \bA' \bQ$, and let us
  write $A$ as $F^{<\gamma}$ and $A'$ as ${F'}^{<\gamma'}$.  Since
  $AP \subseteq A$, either $\gamma$ is indecomposable and
  $P \subseteq A$ by Lemma~\ref{lemma:incl:<}, or $\gamma$ is
  decomposable and $P \subseteq \varepsilon$ by
  Lemma~\ref{lemma:incl:=}.  The latter is impossible by
  Lemma~\ref{lemma:red:eps}.  Hence $\gamma$ is indecomposable, and
  $P = A'Q \subseteq A$.

  If $\gamma' \leq \gamma$, then by Lemma~\ref{lemma:incl:<}
  $F' \subseteq F$ and $Q \subseteq A$, or $F' \not\subseteq F$ and
  $P \subseteq \varepsilon$.  The latter is impossible, as above.  The
  former is impossible, too, because $\bA \bA' \bQ$ is reduced, using
  Definition~\ref{defn:reduced}~(3), (4).  If $\gamma' > \gamma$, then
  $P=A'Q \subseteq \varepsilon$ by Lemma~\ref{lemma:incl:>}, which is
  impossible by Lemma~\ref{lemma:red:eps}.
\end{proof}

\begin{lemma}
  \label{lemma:red:3}
  For every reduced sequence of atoms $\bA \bP$ and for every atom
  $\bB \eqdef {F'}^{<\gamma'}$ with $F' \neq \emptyset$, if $AP = B$
  then $\bP = \varepsilon$.
\end{lemma}
\begin{proof}
  By induction on the number $n \geq 1$ of atoms in $\bA\bP$.  If
  $n=1$, this is vacuous, so let $n \geq 2$.  If $B \subseteq A$, then
  $B \subseteq A \subseteq AP =B$, so $A=B$.  Then $AP=A$, so
  $\bP = \varepsilon$ by Lemma~\ref{lemma:red:3a}.

  We now assume $B \not\subseteq A$, and we will show that this is
  impossible.  By Lemma~\ref{lemma:red:2} applied to $B \subseteq AP$,
  we have $B \subseteq P$.  Since $P \subseteq AP=B$, $B=P$.  By
  induction hypothesis, if we write $\bP$ as $\bA' \bQ$ where $\bA'$
  is an atom, then $\bQ = \varepsilon$, so $\bP = \bA'$.  Then $P=B$
  entails $\bP=\bA'=\bB$ by Lemma~\ref{lemma:red:1}~(3).  Let us write
  $\bA$ as $(F, \gamma)$ and $\bB$ as $(F', \gamma')$.  Since
  $A \subseteq AP = B$, Lemma~\ref{lemma:red:1}~(2) entails that
  $F \subseteq F'$ and $\gamma \leq \gamma'$.  But
  $\bA\bP = (F, \gamma) (F', \gamma')$ is reduced, and this
  contradicts Definition~\ref{defn:reduced}~(2) if $\gamma < \gamma'$,
  (3) if $\gamma=\gamma'$ is indecomposable.  Hence $\gamma=\gamma'$
  is decomposable.  Then
  $AP = F^{<\gamma} {F'}^{<\gamma} \subseteq B = F'^{<\gamma}$ implies
  ${F'}^{<\gamma} \subseteq \varepsilon$ by Lemma~\ref{lemma:incl:=},
  and that is impossible by Lemma~\ref{lemma:red:1}~(1).
%
  %
%
\end{proof}

\begin{proposition}
  \label{prop:reduced}
  For all reduced sequences of atoms $\bP$ and $\bQ$, $P=Q$ if and
  only if $\bP=\bQ$.
\end{proposition}
\begin{proof}
  The `if' direction is trivial.  We show that $P=Q$ implies $\bP=\bQ$
  by induction on the sum $|\bP|+|\bQ|$ of the sizes of $\bP$ and
  $\bQ$, where by
  size we mean number of atoms.  Henceforth we assume $P=Q$.

  If $|\bP|=0$, then $P=Q=\{\epsilon\}$.  By
  Lemma~\ref{lemma:red:eps}, $\bQ = \varepsilon = \bP$.  The situation
  is symmetric if $|\bQ|=0$.  Let us assume $|\bP|, |\bQ| \geq 1$. If
  $|\bQ|=1$, $\bQ$ is an atom $\bB$.  By Lemma~\ref{lemma:red:3}, we
  must have $|\bP|=1$.  Then $P=Q$ implies $\bP=\bQ$ by
  Lemma~\ref{lemma:red:1}~(3).  Similarly if $|\bP|=1$.  The
  interesting case is the remaining one: $|\bP|, |\bQ| \geq 2$.  Let
  us write $\bP$ as $\bA_1 \bA_2 \bP'$ and $\bQ$ as
  $\bB_1 \bB_2 \bQ'$.

  We first claim that if $A_1 \not\subseteq B_1$, then
  $\bP = \bB_2 \bQ'$.  Indeed, under that assumption, and since
  $P = A_1 A_2 P' \subseteq Q = B_1 B_2 Q'$, we obtain
  $P \subseteq B_2 Q'$ by Lemma~\ref{lemma:red:2}.  In turn,
  $B_2 Q' \subseteq B_1 B_2 Q' = Q = P$, so $P = B_2 Q'$.  The
  induction hypothesis then yields $\bP = \bB_2 \bQ'$.  Similarly, if
  $B_1 \not\subseteq A_1$ then $\bQ = \bA_2 \bP'$.

  It follows that we cannot have $A_1 \not\subseteq B_1$ and
  $B_1 \not\subseteq A_1$.  Otherwise,
  $\bQ = \bB_1 \bB_2 \bQ' = \bB_1 \bP = \bB_1 \bA_1 \bA_2 \bP' = \bB_1
  \bA_1 \bQ$, which is impossible since
  $|\bB_1 \bA_1 \bQ| \neq |\bQ|$.

  Hence $A_1$ is included in $B_1$, or conversely.  Without loss of
  generality, let us assume $B_1 \subseteq A_1$.  We claim that, in
  fact, $A_1=B_1$.  We reason by contradiction, and we assume
  $A_1 \not\subseteq B_1$. Then we have seen that $\bP = \bB_2 \bQ'$.
  By syntactic matching, $\bA_1= \bB_2$ (and $\bQ' = \bA_2 \bP'$).
  Let us write $\bB_1$ as $(F, \gamma)$ and $\bB_2$ as
  $(F', \gamma')$.  Since $B_1 \subseteq A_1=B_2$, $F \subseteq F'$
  and $\gamma \leq \gamma'$ by Lemma~\ref{lemma:red:1}~(2).  By
  Definition~\ref{defn:reduced}~(2) and~(3) applied to
  $\bB_1 \bB_2 \bQ'$, it is impossible that $\gamma < \gamma'$, or
  that $\gamma=\gamma'$ is indecomposable.  Hence $\gamma=\gamma'$ is
  decomposable.  Now $Q = B_1 B_2 Q' = F^{<\gamma} B_2 Q'$ is included
  in $P = A_1 A_2 P' = B_2 A_2 P' = {F'}^{<\gamma} A_2 P'$ (since
  $\bA_1=\bB_2$ and $\gamma=\gamma'$), so $B_2 Q' \subseteq A_2 P'$ by
  Lemma~\ref{lemma:incl:=}~(2).
  Since $A_2 P' \subseteq P = B_2 Q'$ (because $\bP = \bB_2 \bQ'$),
  $A_2 P' = B_2 Q'$.  By induction hypothesis,
  $\bA_2 \bP' = \bB_2 \bQ'$, so $\bA_2 \bP' = \bP$. This is impossible
  since $\bP = \bA_1 \bA_2 \bP'$.
  Having reached a contradiction, we
  conclude that $A_1=B_1$, so $\bA_1 = \bB_1$ by
  Lemma~\ref{lemma:red:1}~(3).

  We now claim that $A_2 P' \subseteq B_2 Q'$.  We know that
  $P = A_1 A_2 P'$ is included in $Q = B_1 B_2 Q' = A_1 B_2 Q'$ (since
  $A_1=B_1$).  Let us write $\bA_1$ as $(F, \gamma)$ and $\bA_2$ as
  $(F', \gamma')$.  If $\gamma$ is decomposable, then
  $A_2 P' \subseteq B_2 Q'$ by Lemma~\ref{lemma:incl:=}~(2).  Let
  therefore $\gamma$ be indecomposable.  Since $\bA_1 \bA_2 \bP'$ is
  reduced, we cannot have $F \supseteq F'$ and $\gamma \geq \gamma'$,
  by Definition~\ref{defn:reduced}~(3) and~(4), so
  $A_2 \not\subseteq A_1$, by Lemma~\ref{lemma:red:1}~(2).  Then
  $A_2 P' \subseteq A_1 A_2 P' \subseteq A_1 B_2 Q'$,
  so $A_2 P' \subseteq B_2 Q'$ by Lemma~\ref{lemma:red:2}.



  Symmetrically, $B_2 Q'$ is included in $A_2 P'$, so
  $A_2 P' = B_2 Q'$. By the induction hypothesis
  $\bA_2 \bP' = \bB_2 \bQ'$. We remember that $\bA_1 = \bB_1$, so
  $\bP = \bQ$. \end{proof}

We can always rewrite any transfinite product into a reduced product
with the same semantics, using Lemma~\ref{lemma:prod:equal}, whence the
following.
\begin{corollary}
  \label{corl:reduced}
  Every transfinite product is equal to $P$ for some unique reduced
  sequence of atoms $\bP$.
\end{corollary}
Corollary~\ref{corl:reduced} allows us to conflate the notions of
transfinite product and of reduced sequence of atoms.  By abuse of
language, we will call \emph{reduced product} any transfinite product
$P$ written in such a way that $\bP$ is reduced.  A \emph{reduced
  $\alpha$-product} is an $\alpha$-product
that is reduced in this sense.

\section{Well-foundedness}
\label{sec:well-foundedness}

The \emph{rank} (or \emph{height}) of an element $x$ in a well-founded
poset $P$ is defined by well-founded induction as the least ordinal
strictly larger than the ranks of all elements $y < x$. We write
$||F||$ for the rank of $F$ in the lattice of closed subsets of a
Noetherian space $X$. For any Noetherian space $F$, $||F||$ is the
\emph{stature} of $F$ \cite{GLL:stature}, generalizing the notion of
the same name on wqos \cite{BG:stature}.

Given any two ordinals
$\alpha \eqdef \omega^{\alpha_1} + \cdots + \omega^{\alpha_m}$ and
$\beta \eqdef \omega^{\beta_1} + \cdots + \omega^{\beta_n}$ in Cantor
normal form, their \emph{natural sum} $\alpha \oplus \beta$ is defined
as $\omega^{\gamma_1} + \cdots + \omega^{\gamma_{m+n}}$, where
$\gamma_1 \geq \cdots \geq \gamma_{m+n}$ is the list obtained by
sorting the list
$\alpha_1, \cdots, \alpha_m, \beta_1, \cdots, \beta_n$ in decreasing
order. This operation is associative and commutative, and strictly
monotonic in both arguments.

An ordinal $\delta$ is \emph{critical} if and only if
$\omega^\delta = \delta$. For every ordinal $\alpha$, let
$\adjsw\alpha$ be $\alpha+1$ if $\alpha = \delta+n$ for some critical
ordinal $\delta$ and some $n \in \nat$, and $\alpha$ otherwise. Then
$\alpha < \omega^{\adjsw\alpha}$, and $\alpha \mapsto \adjsw \alpha$
is strictly monotonic \cite[Lemmata~12.3 and~12.4]{GLL:stature}. For
every proper bound $\gamma$, we define $\flt\gamma$ as the rank of
$\gamma$ in the poset of all proper bounds less than or equal to
$\gamma$. Explicitly, and writing $n$ for a natural number and
$\lambda$ for a limit ordinal, $\flt {\omega^n} = 2n-2$ and
$\flt {\omega^n+1} = 2n-1$ if $n \geq 1$,
$\flt {\omega^{\lambda+n}} = \lambda+2n$, and
$\flt {\omega^{\lambda+n}+1} = \lambda+2n+1$.

\begin{definition}
  \label{defn:phi}
  Let $X$ be a Noetherian space.  For every atom $F^{<\gamma}$,
  let $\varphi (F^{<\gamma})$ be $||F||$ if $\gamma = \omega^0+1$, and
  $\omega^{\adjsw{(||F|| \oplus \flt\gamma)}}$ otherwise.  For every
  reduced product $P \eqdef A_1 \cdots A_n$, let
  $\varphi (P) \eqdef \bigoplus_{i=1}^n \varphi (A_i)$.
\end{definition}

\begin{proposition}
  \label{prop:phi}
  Let $X$ be a Noetherian space.  For all reduced products $P$ and
  $P'$, $P \subseteq P'$ implies $\varphi (P) \leq \varphi (P')$, and
  $P \subsetneq P'$ implies $\varphi (P) < \varphi (P')$.
\end{proposition}
\begin{proof}
  We first claim that $\varphi$ is strictly monotonic on reduced
  atoms: for all atoms $F^{<\gamma}$ and ${F'}^{<\gamma'}$ with
  $F, F' \neq \emptyset$ and $F^{<\gamma} \subsetneq {F'}^{<\gamma'}$,
  $\varphi (F^{<\gamma}) < \varphi ({F'}^{<\gamma'})$.  By
  Lemma~\ref{lemma:red:1}, $F \subseteq F'$ and $\gamma \leq \gamma'$,
  and not both are equalities.  If $\gamma = \gamma' = \omega^0+1$ and
  $F \subsetneq F'$, then
  $\varphi (F^{<\gamma}) = ||F|| < ||F'|| = \varphi
  ({F'}^{<\gamma'})$.  If $\gamma = \omega^0+1 < \gamma'$, then
  $\varphi (F^{<\gamma}) = ||F|| \leq ||F|| \oplus \flt\gamma < ||F'||
  \oplus \flt{\gamma'} < \omega^{\adjsw{(||F'|| \oplus
      \flt{\gamma'})}} = \varphi ({F'}^{<\gamma'})$.  If both $\gamma$
  and $\gamma'$ are proper, then
  $\varphi (F^{<\gamma}) = \omega^{\adjsw{(||F|| \oplus \flt\gamma)}}
  < \omega^{\adjsw{(||F'|| \oplus \flt{\gamma'})}} = \varphi
  ({F'}^{<\gamma'})$ since $\oplus$, $\flt\_$, $\adjsw\_$ and $||\_||$
  are strictly monotonic.

  
  We prove the proposition by induction on the sum of the lengths of
  $P$ and $P'$. Let $P \subseteq P'$. If $P = \varepsilon$, then
  $\varphi (P)=0 \leq \varphi (P')$.  If additionally
  $P' \neq \varepsilon$, then $\varphi (P')$ is the natural sum of at
  least one term, and all those terms are non-zero: they are of the
  form $\varphi (F^{<\gamma})$, and
  when $\gamma = \omega^0+1$, $\varphi (F^{<\gamma}) = ||F|| \neq 0$
  since $F \neq \emptyset$, otherwise
  $\varphi (F^{<\gamma})$ is a power of $\omega$.

  We now assume that $P \neq \varepsilon$. We write $P$ as
  $A_1 \cdots A_n$, where $A_1$, \ldots, $A_n$ are atoms and
  $n \geq 1$. Let also $Q \eqdef A_2 \cdots A_n$, so $P = A_1 Q$.
  Since $P \subseteq P'$, $P' \neq \varepsilon$, so
  $P'=A'_1 Q'$ for some atom $A'_1$ and some (reduced)
  product $Q'$.  We write $A'_1$ as ${F'}^{<\gamma'}$, and
  $A_i$ as ${F_i}^{<\gamma_i}$ for each $i$, $1\leq i\leq n$.

  If $A_1 \not\subseteq A'_1$, then Lemma~\ref{lemma:red:2} entails
  that $P = A_1 Q \subseteq Q'$.  By induction hypothesis,
  $\varphi (P) \leq \varphi (Q')$.  Now
  $\varphi (P') = \varphi (A'_1) \oplus \varphi (Q') > \varphi (Q')$,
  so $\varphi (P) < \varphi (P')$.  We turn to the other cases: from
  now on, $A_1 \subseteq A'_1$.

  \emph{If $\gamma'$ is indecomposable.}  We say that $A_i$ is
  \emph{small} if $A_i \subseteq A'_1$, equivalently
  $F_i \subseteq F'$ and $\gamma_i \leq \gamma'$, by
  Lemma~\ref{lemma:red:1}~(2).  $A_1$ is small.  Let $k$ be largest
  such that $A_1$, \ldots, $A_k$ are small, and
  $R \eqdef A_{k+1} \cdots A_n$.  Then $R \subseteq Q'$.  This is
  clear if $R = \varepsilon$; otherwise $k < n$,
  $A_{k+1} \not\subseteq A'_1$, then
  $R \subseteq P \subseteq P' = A'_1 Q'$ implies $R \subseteq Q'$ by
  Lemma~\ref{lemma:red:2}.  By the induction hypothesis,
  $\varphi (R) \leq \varphi (Q')$.
  
  If some $A_i$ with $1\leq i\leq k$ is equal to
  $A'_1 = {F'}^{<\gamma'}$, then $\gamma_i=\gamma'$ is indecomposable,
  and $F_i = F'$.  We cannot have $i \geq 2$, since that would
  contradict Definition~\ref{defn:reduced}~(2) or~(3) at positions
  $i-1$ and $i$; similarly, $i \leq k-1$ would contradict
  Definition~\ref{defn:reduced}~(4) or~(3) at positions $i$ and
  $i+1$. Hence $k=1$. Then
  $\varphi (P) = \varphi (A_1) \oplus \varphi (R) \leq \varphi (A'_1)
  \oplus \varphi (Q') = \varphi (P')$.  Additionally, if $P \neq P'$,
  then (since $k=1$ and $A_1=A'_1$) $R$ is strictly included in $Q'$,
  so $\varphi (R) < \varphi (Q')$ by induction hypothesis, from which
  $\varphi (P) < \varphi (P')$ follows.
  
  Otherwise, $k \geq 1$ and $A_1$, \ldots, $A_k$ are all strictly
  included in $A'_1$.  Then $\varphi (A_i) < \varphi (A'_1)$ for every
  $i$ with $1\leq i\leq k$.  Since $\gamma'$ is indecomposable,
  $\gamma' \neq \omega^0+1$, so
  $\varphi (A'_1) = \omega^{\adjsw{(||F'|| \oplus \flt{\gamma'})}}$.
  The latter is indecomposable, so
  $\varphi (A_1) \oplus \cdots \oplus \varphi (A_k) < \varphi (A'_1)$.
  Together with $\varphi (R) \leq \varphi (Q')$, this implies that
  $\varphi (P) = \varphi (A_1) \oplus \cdots \oplus \varphi (A_k)
  \oplus \varphi (R) < \varphi (A'_1) \oplus \varphi (Q') = \varphi
  (P')$.

  \emph{If $\gamma'$ is decomposable.}  In that case, we say that
  $A_i$ is \emph{small} if and only if $F_i \subseteq F'$ and
  $\gamma_i < \gamma'$ (not $\leq$).  Let $k$ be largest such that
  $A_1$, \ldots, $A_k$ are small, and $R \eqdef A_{k+1} \cdots A_n$.
  $A_1$, \ldots, $A_k$ are all strictly included in $A'_1$, so
  $\varphi (A_i) < \varphi (A'_1)$ for every $i$ with $1\leq i\leq k$.
  If $k \geq 1$, then $\gamma_1 < \gamma'$, so
  $\gamma' \neq \omega^0+1$, and therefore $\varphi (A'_1)$ is
  indecomposable.  Hence
  $\varphi (A_1) \oplus \cdots \oplus \varphi (A_k) < \varphi (A'_1)$.
  If $k=0$, the same inequality holds, vacuously.
  
  If $R = \varepsilon$, then
  $\varphi (P) = \varphi (A_1) \oplus \cdots \oplus \varphi (A_k) <
  \varphi (A'_1) \leq \varphi (P')$.  We now assume
  $R \neq \varepsilon$.  Then, $k < n$, and $A_{k+1}$ is not small.

  If $A_{k+1} \not\subseteq A'_1$, then
  $R \subseteq P \subseteq P' = A'_1 Q'$ implies $R \subseteq Q'$ by
  Lemma~\ref{lemma:red:2}, hence $\varphi (R) \leq \varphi (Q')$ by
  induction hypothesis.  Then
  $\varphi (P) = \varphi (A_1) \oplus \cdots \oplus \varphi (A_k)
  \oplus \varphi (R) < \varphi (A'_1) \oplus \varphi (Q') = \varphi
  (P')$.

  There remains one case, where $A_{k+1} \subseteq A'_1$ and $A_{k+1}$
  is not small.  Then $F_{k+1} \subseteq F'$ and
  $\gamma_{k+1}=\gamma'$, which is decomposable.  Let
  $R' \eqdef A_{k+2} \cdots A_n$.
  $R =A_{k+1} R' \subseteq P \subseteq P' = A'_1 Q'$ implies
  $R' \subseteq Q'$ by Lemma~\ref{lemma:incl:=}.  By induction
  hypothesis, $\varphi (R') \leq \varphi (Q')$.  If $k=0$, then
  $P = A_1 R'$, $P' = A'_1 Q'$, $A_1 \subseteq A'_1$, and
  $R' \subseteq Q'$, so
  $\varphi (P) = \varphi (A_1) \oplus \varphi (R') \leq \varphi (A'_1)
  \oplus \varphi (Q') = \varphi (P')$; additionally, if
  $P \subsetneq P'$ then $A_1 \subsetneq A'_1$ or $R' \subsetneq Q'$,
  which implies $\varphi (P) < \varphi (P')$.  Let us now assume
  $k \geq 1$.  We claim that the inclusion $A_{k+1} \subseteq A'_1$ is
  strict: if $A_{k+1}=A'_1$, then $A_k \subseteq A'_1 = A_{k+1}$ and
  $\gamma_k < \gamma' = \gamma_{k+1}$, contradicting
  Definition~\ref{defn:reduced}~(2).  Hence $A_1$, \ldots, $A_k$, and
  also $A_{k+1}$, are strictly included in $A'_1$.  We recall that,
  since $k \geq 1$, $\varphi (A'_1)$ is indecomposable, so
  $\varphi (P) = \varphi (A_1) \oplus \cdots \oplus \varphi (A_k)
  \oplus \varphi (A_{k+1}) \oplus \varphi (R') < \varphi (A'_1) \oplus
  \varphi (Q') = \varphi (P')$.
\end{proof}
\begin{corollary}
  \label{corl:wf}
  Let $X$ be a Noetherian space.  The inclusion ordering on
  transfinite products on $X$ is well-founded.  Additionally, the
  ordinal rank of $X^{<\alpha}$ in the poset of all $\alpha$-products
  is at most $\omega^{\adjsw{(||X|| \oplus \flt\alpha)}}$, for every
  bound $\alpha$.
\end{corollary}

\section{Intersections of transfinite products}
\label{sec:inters-transf-prod}

\begin{lemma}
  \label{lemma:inter}
  Let $X$ be a topological space.  The intersection of two transfinite
  products satisfies the following properties:
  \begin{enumerate}
  \item $\varepsilon \cap P' = P \cap \varepsilon = \varepsilon$.
  \item If $\gamma < \gamma'$ or if $\gamma=\gamma'$ is
    indecomposable, then $F^{<\gamma} P \cap {F'}^{<\gamma'} P'$ is
    equal to the union of
    $(F \cap F')^{<\gamma}(P \cap {F'}^{<\gamma'} P')$ and of
    $(F \cap F')^{<\gamma}(F^{<\gamma} P \cap P')$.
  \item If $\gamma=\gamma'$ is decomposable, say
    $\gamma = \gamma' = \omega^\beta+1$, then
    $F^{<\gamma} P \cap {F'}^{<\gamma'} P'$ is the union of
    $(F \cap F')^{<\gamma} (P \cap P')$, of
    $(F \cap F')^{<\omega^\beta}(P \cap {F'}^{<\gamma'} P')$, and of
    $(F \cap F')^{<\omega^\beta}(F^{<\gamma} P \cap P')$.
  \end{enumerate}
\end{lemma}
\begin{proof}
  (1) is clear.  Let us deal with the left to right inclusions for the
  other cases, as a first step.
  For every $w \in F^{<\gamma} P \cap {F'}^{<\gamma'} P'$, let us
  write $w$ as $uv$ where $u \in F^{<\gamma}$, $v \in P$ and also as
  $u'v'$ where $u' \in {F'}^{<\gamma'}$, $v' \in P'$.
    
  (2) If $|u| \leq |u'|$, $u$ is a prefix of $u'$, so its
  letters are not just in $F$, but also in $F'$. Therefore
  $u \in (F \cap F')^{<\gamma}$. Also, $v$ is in $P$, and is a suffix
  of $w$.  In particular, $v \leq_* w$, 
  so $v \in {F'}^{<\gamma'} P'$ by
  Lemma~\ref{lemma:prod}.  If $|u'| \leq |u|$, then symmetrically
  $u'$ is in $(F \cap F')^{<\gamma}$ and $v'$ is in $P'$ and in
  $F^{<\gamma} P$.

  (3) If $|u|=|u'| = \omega^\beta$, then $u=u'$ is in
  $(F \cap F')^{<\gamma}$, and $v=v'$ is in $P \cap P'$, so
  $w \in (F \cap F')^{<\gamma} (P \cap P')$.  Otherwise, $w$ is in
  $F^{<\omega^\beta} P$ and in ${F'}^{<\gamma'} P'$, or in
  $F^{<\gamma} P$ and in ${F'}^{<\omega^\beta} P'$.  In the first
  case, by (2) with $\omega^\beta$ in lieu of $\gamma$, $w$ is in
  $(F \cap F')^{<\omega^\beta}(P \cap {F'}^{\gamma'} P') \cup (F \cap
  F')^{<\omega^\beta}(F^{<\omega^\beta} P \cap P')$, hence in
  $(F \cap F')^{<\omega^\beta}(P \cap {F'}^{\gamma'} P') \cup (F \cap
  F')^{<\omega^\beta}(F^{<\gamma} P \cap P')$.  In the second case,
  a similar argument leads to the same result.
    
  We now deal with the right to left inclusions.

  (2) $(F \cap F')^{<\gamma}(P \cap {F'}^{<\gamma'} P')$ is included
  both in $F^{<\gamma} P$ (because $F \cap F' \subseteq F$) and in
  $(F \cap F')^{<\gamma} {F'}^{<\gamma'} P' = {F'}^{\gamma'} P'$, by
  Lemma~\ref{lemma:prod:equal}~(1).  Similarly for
  $(F \cap F')^{<\gamma}(F^{<\gamma} P \cap P')$.

  (3) $(F \cap F')^{<\gamma} (P \cap P')$ is included both in
  $F^{<\gamma} P$ and in ${F'}^{<\gamma} P' = {F'}^{<\gamma'} P'$.
  $(F \cap F')^{<\omega^\beta}(P \cap {F'}^{\gamma'} P')$ is included
  both in
  $(F \cap F')^{<\omega^\beta} P \subseteq F^{<\omega^\beta} P
  \subseteq F^{<\gamma} P$ and in
  $(F \cap F')^{<\omega^\beta} {F'}^{\gamma'} P' = {F'}^{\gamma'} P'$
  (by Lemma~\ref{lemma:prod:equal}~(1)).  Finally,
  $(F \cap F')^{<\omega^\beta}(F^{<\gamma} P \cap P')$ is included
  both in $(F \cap F')^{<\omega^\beta} F^{<\gamma} P = F^{<\gamma} P$
  (by Lemma~\ref{lemma:prod:equal}~(1)) and in
  $(F \cap F')^{<\omega^\beta} P' \subseteq {F'}^{<\omega^\beta} P'
  \subseteq {F'}^{<\gamma'} P'$.
\end{proof}

\begin{corollary}
  \label{corl:inter}
  Let $X$ be a Noetherian space.  The intersection of any two
  transfinite products is a finite union of transfinite products.
\end{corollary}
\begin{proof}
  By induction on the sum of their sizes, using
  Lemma~\ref{lemma:inter}.  The case where one of them is
  $\varepsilon$ is obvious, so we deal with the intersection of two
  transfinite products $F^{<\gamma} P$ and ${F'}^{<\gamma'} P'$.
  Without loss of generality, $\gamma \leq \gamma'$.
  
  By induction hypothesis, $P \cap {F'}^{\gamma'} P'$,
  $F^{<\gamma} P \cap P'$ and $P \cap P'$ are
  finite unions of transfinite products, say $\bigcup_i P_i$,
  $\bigcup_j Q_j$, and $\bigcup_k R_k$ respectively.

  The intersection $F^{<\gamma} P \cap {F'}^{<\gamma'} P'$ can then be
  expressed as
  $\bigcup_i (F \cap F')^{<\gamma} P_i \cup \bigcup_j (F \cap
  F')^{<\gamma} Q_j$ when $\gamma < \gamma'$ or if $\gamma = \gamma'$
  is indecomposable.  If $\gamma=\gamma'$ is of the form
  $\omega^\beta+1$, then it can be expressed as the union
  $\bigcup_k (F \cap F')^{<\gamma} R_k \cup \bigcup_i (F \cap
  F')^{<\omega^\beta} P_i \cup \bigcup_j (F \cap F')^{<\omega^\beta}
  Q_j$.

  This is a finite union of transfinite products, except when
  $\gamma = \omega^0+1$.  In that case, we need to refine the
  expressions above.  If $F \cap F'$ is empty, then
  $(F \cap F')^{<\gamma} 
  = \varepsilon$, so
  $F^{<\gamma} P \cap {F'}^{<\gamma'} P' = \bigcup_i P_i \cup
  \bigcup_j Q_j$ when $\omega^0+1 < \gamma'$, and
  $F^{<\gamma} P \cap {F'}^{<\gamma'} P' = \bigcup_k R_k \cup
  \bigcup_i P_i \cup \bigcup_j Q_j$ otherwise.

  We now assume that $F \cap F' \neq \emptyset$. $F$ and $F'$ are
  irreducible closed, and since $X$ is Noetherian, $F \cap F'$ is a
  finite union $\bigcup_{\ell=1}^n C_\ell$ of irreducible closed
  subsets.  Since $F \cap F' \neq \emptyset$, $n$ is non-zero.  Then
  $(F \cap F')^{<\gamma} = (F \cap F')^? = \bigcup_{\ell=1}^n
  C_\ell^?$ (an equality that would fail if $n$ were zero), and
  $(F \cap F')^{<\omega^0} = \varepsilon$.  Therefore
  $F^{<\gamma} P \cap {F'}^{<\gamma'} P'$ is equal to
  $\bigcup_{i, \ell} C_\ell^?  P_i \cup \bigcup_{j, \ell} C_\ell^?
  Q_j$ when $\omega^0+1 < \gamma'$, or to
  $\bigcup_{k, \ell} C_\ell^? R_k \cup \bigcup_i P_i \cup \bigcup_j
  Q_j$ if $\gamma' = \omega^0+1$.
\end{proof}

We can now use Proposition~\ref{prop:S}. If $X$ is Noetherian, the
class $\mathcal P$ of $\alpha$-products is well-founded under
inclusion by Corollary~\ref{corl:wf}, $X^{<\alpha}$ is a finite
union of elements of $\mathcal P$ since it is, in fact, an
$\alpha$-product.  The intersection of any two elements of $\mathcal P$
is a finite union of elements of $\mathcal P$ by
Corollary~\ref{corl:inter}, and $\mathcal P$ is irredundant by
Proposition~\ref{prop:irred}. Moreover, the regular subword topology
on $X^{<\alpha}$ is the coarsest that makes every element of
$\mathcal P$ a closed set, by Proposition~\ref{prop:alpha:prod}. Since
every set of transfinite words on $X$ is included in $X^{<\alpha}$ for
some bound $\alpha$, we obtain the following. \begin{thm}
  \label{thm:Xalpha}
  For every Noetherian space $X$, every space $Y$ of transfinite words
  on $X$ is Noetherian in the regular subword topology.  For every
  bound $\alpha$, the irreducible closed subsets of $X^{<\alpha}$ are
  the $\alpha$-products.
\end{thm}

Given any non-empty Noetherian space $Z$, $Z$ has finitely many
maximal irreducible closed subsets $C_1, \cdots, C_n \in \Sober Z$
($n \geq 1$).  The \emph{reduced sobrification rank} $\rsob Z$ is the
maximum of the ranks of $C_1$, \ldots, $C_n$ in $\Sober Z$
\cite[Lemma~4.2]{GLL:stature}; then
$||Z|| \leq \omega^{\rsob Z} \otimes n$
\cite[Proposition~4.5]{GLL:stature}.
For $Z \eqdef X^{<\alpha}$, $n=1$ since $X^{<\alpha}$ is the unique
largest element of $\Sober {(X^{<\alpha})}$, so the $\strut\otimes n$
part vanishes. Corollary~\ref{corl:wf} then gives us the following
quantitative information.
\begin{proposition}
  \label{prop:stature}
  For every Noetherian space $X$, for every bound $\alpha$, we have
  $\rsob {X^{<\alpha}} \leq \omega^{\adjsw{(||X|| \oplus
      \flt\alpha)}}$ and
  $||X^{<\alpha}|| \leq \omega^{\omega^{\adjsw{(||X|| \oplus
        \flt\alpha)}}}$.
\end{proposition}
It is not our purpose to give an exact formula for
$\rsob {X^{<\alpha}}$ and $||X^{<\alpha}||$ here. When
$\alpha = \omega^1$, $X^{<\alpha}=X^*$, $\flt\alpha = 0$, and
Proposition~\ref{prop:stature} implies
$\rsob {X^*} \leq \omega^{\adjsw{||X||}}$ and
$||X^*|| \leq \omega^{\omega^{\adjsw{||X||}}}$. The former is
essentially optimal, since if $X \neq \emptyset$, then
$\rsob {X^*} = \omega^{\adjsw{||X||}}$
\cite[Theorem~12.13]{GLL:stature}, while the latter is close to
optimal: if $||X||$ is infinite,
$||X^*|| = \omega^{\omega^{\adjsw{||X||}}}$
\cite[Theorem~12.23]{GLL:stature}, generalizing Schmidt's formula on
well-partial-orders \cite[Theorem~9]{Schmidt:maxot}.  But when
$||X||$ is finite and non-empty,
$||X^*||=\omega^{\omega^{||X||-1}} < \omega^{\omega^{\adjsw{||X||}}}$.

\section{The specialization ordering}
\label{sec:spec-order-xalpha}

A non-empty transfinite word $w$ is indecomposable if and only if for
every way of writing $w$ as $uv$ where $v \neq \epsilon$, we have
$w \leq_* v$.  Pouzet \cite{Pouzet:pmo}, confirming a conjecture of
Jullien \cite{Jullien:PhD}, shows that $\leq$ is a
better-quasi-ordering if and only if every transfinite word (of
countable length) is a concatenation of finitely many indecomposable
words.

Given a topological space $X$ and an ordinal $\alpha$, let us call $w \in X^{<\alpha}$ \emph{topologically
  indecomposable} if and only $w \neq \epsilon$ and, for every way of
writing $w$ as $uv$ where $v \neq \epsilon$, $w$ is in the closure
$\overline v$ of $v$ in $X^{<\alpha}$. (We write $\overline v$ instead
of the more cumbersome notation $cl (\{v\})$.) When $w = uv$, we have
$v \leq_* w$, so $v \in \overline w$ by Lemma~\ref{lemma:prod}.
Hence $w$ is topologically indecomposable if and only if for every way
of writing $w$ as $uv$ where $v \neq \epsilon$,
$\overline v = \overline w$.

\begin{lemma}
  \label{lemma:indecomp}
  Let $X$ be a topological space.  Every indecomposable transfinite
  word $w$ on $X$ is topologically indecomposable.
\end{lemma}
\begin{proof}
  Whenever $w=uv$ with $v \neq \epsilon$,
  $w \leq_* v$; so $w \in \overline v$ by Lemma~\ref{lemma:prod}.
\end{proof}

\begin{lemma}
  \label{lemma:indecomp:decomp}
  Every transfinite word $w \in X^{<\alpha}$, where $X$ is Noetherian
  and $\alpha$ is a bound, can be written as a finite concatenation of
  topologically indecomposable transfinite words.
\end{lemma}
\begin{proof}
  By induction on $|w|$. The claim is clear if
  $w = \epsilon$.  Otherwise, for every $\beta < |w|$, we
  write $w$ as $u_\beta v_\beta$ where $u_\beta$ is its
  prefix of length $\beta$, and $v_\beta \neq \epsilon$. For all
  ordinals $\beta < \gamma < |w|$, $v_\gamma \leq_* v_\beta$, so
  $v_\gamma \in \overline {v_\beta}$ by Lemma~\ref{lemma:prod},
  and therefore $\overline {v_\gamma} \subseteq \overline {v_\beta}$.
  Hence the closed sets $\overline {v_\beta}$ with $\beta < |w|$ form
  a chain.  Since $X^{<\alpha}$ is Noetherian
  (Theorem~\ref{thm:Xalpha}), inclusion is well-founded on its closed
  subsets, so there is a non-empty suffix
  $v_\beta$ of $w$ whose closure $\overline {v_\beta}$ is smallest.

  We claim that $v_\beta$ is topologically indecomposable.  Let us
  write $v_\beta$ as $uv$ with $v \neq \epsilon$.  Then
  $v = v_\gamma$, where $\gamma \eqdef \beta + |u|$, and
  $u_\gamma = u_\beta u$.  Since $\overline {v_\beta}$ is least and
  $\overline {v_\gamma} \subseteq \overline {v_\beta}$, owing to the
  fact that $\gamma \geq \beta$, we deduce that
  $\overline v = \overline {v_\gamma} = \overline {v_\beta}$.
  
  By induction hypothesis (indeed, $|u_\beta|=\beta < |w|$), $u_\beta$
  is a finite concatenation
  of topologically indecomposable words; hence so is
  $w = u_\beta v_\beta$.
\end{proof}

\begin{lemma}
  \label{lemma:indecomp:length}
  The length of a topologically indecomposable transfinite word is
  indecomposable.
\end{lemma}
\begin{proof}
  Let $w$ be topologically indecomposable in $X^{<\alpha}$,
  and let us assume that $|w|$ is decomposable.  We write $|w|$ as
  $\beta+\gamma$, where $0 < \beta, \gamma < |w|$, and then $w$ as
  $uv$ where $|u|=\beta$. We have $|v| = \gamma \neq 0$, so
  $v \neq \epsilon$. Since $w$ is topologically indecomposable, $w$ is
  in $\overline v$. Now $v$ is in $X^{<\gamma+1}$, which is closed, so
  $\overline v \subseteq X^{<\gamma+1}$. Thus 
  $w$ is in $X^{<\gamma+1}$, which is impossible since $\gamma < |w|$.
\end{proof}

\begin{lemma}
  \label{lemma:indecomp:split}
  Let $X$ be a Noetherian space and $\alpha$ be a bound.  For every
  non-empty transfinite word $w \in X^{<\alpha}$, $w$ is topologically
  indecomposable if and only if for every closed subset of
  $X^{<\alpha}$ that is the concatenation $\mathcal C_1 \mathcal C_2$
  of two closed subsets of $X^{\alpha}$, if
  $w \in \mathcal C_1 \mathcal C_2$ then $w$ is in $\mathcal C_1$ or
  in $\mathcal C_2$.
\end{lemma}
\begin{proof}
  Let $w$ be topologically indecomposable. Since
  $w \in \mathcal C_1 \mathcal C_2$, we can write $w$ as $uv$ where
  $u \in \mathcal C_1$ and $v \in \mathcal C_2$.  If $v = \epsilon$,
  then $w=u$ is in $\mathcal C_1$.  Otherwise, since $w$ is
  topologically indecomposable, $\overline v = \overline w$.  But
  $\overline v \subseteq \mathcal C_2$, so
  $w \in \overline w \subseteq \mathcal C_2$.

  Conversely, let us assume that
  $w \in \mathcal C_1 \mathcal C_2$ implies $w \in \mathcal C_1$ or
  $w \in \mathcal C_2$ for every concatenation of two closed sets
  $\mathcal C_1 \mathcal C_2$.  For any way of writing $w$ as $uv$
  with $v \neq \emptyset$, we let $\mathcal C_1 \eqdef \overline u$
  and $\mathcal C_2 \eqdef \overline v$.  Then $w \in \overline u$ or
  $w \in \overline v$.  Let $\beta \eqdef |u|$.  We note that
  $\beta < |w|$, since $|w| = \beta + |v|$ and $v \neq \epsilon$.
  Clearly, $u$ is in the closed set $X^{<\beta+1}$, so
  $\overline u \subseteq X^{<\beta+1}$.  If $w \in \overline u$, then
  $w$ is in $X^{<\beta+1}$, so $|w| \leq \beta < |w|$, which is
  impossible.  Therefore $w$ is in $\overline v$.
\end{proof}


Let us write $\Img w$ for the set of letters in a transfinite word
$w$.  We call \emph{support} $\supp w$ of $w$ the closure
$cl (\Img w)$ of $\Img w$.

\begin{lemma}
  \label{lemma:indecomp:cl}
  Let $X$ be Noetherian, and $\alpha$ be a bound.
  For every topologically indecomposable word $w \in
  X^{<\alpha}$, $\overline w = {(\supp w)}^{<\gamma+1}$, where
  $\gamma \eqdef |w|$.
\end{lemma}
\begin{proof}
  Let $F \eqdef \supp w$.  Clearly, $w$ is in $F^{<\gamma+1}$, so
  $\overline w \subseteq F^{<\gamma+1}$.  Conversely, $\overline w$ is
  irreducible closed, hence is an $\alpha$-product
  $P \eqdef A_1 \cdots A_n$, by Theorem~\ref{thm:Xalpha}.  Since
  $w \in P$ is topologically indecomposable hence non-empty,
  $n \geq 1$.  Using Lemma~\ref{lemma:indecomp:split},
  $w$ is in some $A_i$. Let us write
  $A_i$ as ${F'}^{<\gamma'}$.  Since $|w| = \gamma$, 
  $\gamma < \gamma'$.  Every letter of $w$ is in $F'$, so
  $\Img w \subseteq F'$, and taking closures,
  $F \subseteq F'$.  Then $F^{<\gamma+1} \subseteq {F'}^{<\gamma'}$,
  so $F^{<\gamma+1} \subseteq P = \overline w$.
\end{proof}

\begin{lemma}
  \label{lemma:cat:closed}
  Let $X$ be a Noetherian space, $\beta$ be an ordinal, and $\alpha$
  be $\omega^\beta$ or $\omega^\beta+1$.  The map $cat$ is closed and
  continuous from $X^{<\omega^\beta} \times X^{<\alpha}$ to
  $X^{<\alpha}$.
\end{lemma}
\begin{proof}
  It is well-defined and continuous by Lemma~\ref{lemma:cat:cont}.  By
  Theorem~\ref{thm:Xalpha}, the closed subsets of $X^{<\omega^\beta}$
  (resp., $X^{<\alpha}$) are the finite unions of
  $\omega^\beta$-products (resp., $\alpha$-products).  Given any two
  such closed sets expressed as finite unions of such products
  $\mathcal C_1 \eqdef \bigcup_{i=1}^m P_i$ and
  $\mathcal C_2 \eqdef \bigcup_{j=1}^n Q_j$, the image
  $\mathcal C_1 \mathcal C_2$ of $\mathcal C_1 \times \mathcal C_2$ by
  $cat$ is the (finite) union over all $i$ and $j$ of the products
  $P_i Q_j$.
\end{proof}

\begin{lemma}
  \label{lemma:cl:cat}
  Let $X$ be a Noetherian space, and $\alpha$ be a bound. For all sets
  of transfinite words $A$ and $B$ such that
  $AB \subseteq X^{<\alpha}$, $cl (AB)=cl (A) cl (B)$.
\end{lemma}
\begin{proof}
  We use Lemma~\ref{lemma:cat:closed}.  If $\alpha$ is indecomposable,
  $cat$ is continuous from $X^{<\alpha} \times X^{<\alpha}$ to
  $X^{<\alpha}$, so $cl (A) cl (B) \subseteq cl (AB)$; $cat$ is
  closed, so $cl (A) cl (B)$ is closed, and contains $AB$, so it
  contains $cl (AB)$.

  If $\alpha = \omega^\beta+1$, then either
  $B \subseteq \{\epsilon\}$, in which case
  $cl (AB) = cl (A) = cl (A) cl (B)$; or
  $A \subseteq X^{<\omega^\beta}$ (else we could pick $u \in A$ of
  length $\omega^\beta$, $v \neq \epsilon$ in $B$, and then
  $uv \in AB$ would not be in $X^{<\alpha}$), then we reason as above,
  using the fact that $cat$ is closed and continuous from
  $X^{<\omega^\beta} \times X^{<\alpha}$ to $X^{<\alpha}$.
\end{proof}

\begin{proposition}
  \label{prop:cl:cat}
  For every Noetherian space $X$ and every bound $\alpha$, the regular
  subword topology on $X^{<\alpha}$ is the coarsest one such that
  $F^{<\gamma}$ is closed for every closed subset $F$ of $X$ and every
  ordinal $\gamma \leq \alpha$, and such that
  $\mathcal C_1 \mathcal C_2$ is closed for all closed subsets
  $\mathcal C_1$ and $\mathcal C_2$ such that
  $\mathcal C_1 \mathcal C_2 \subseteq X^{<\alpha}$.
\end{proposition}
\begin{proof}
  Let us call admissible any topology $\tau$ containing the sets
  $F^{<\gamma}$ as closed sets, and closed under concatenations
  $\mathcal C_1 \mathcal C_2$, as described above.  The regular
  subword topology is admissible, since
  $\mathcal C_1 \mathcal C_2 = cl (\mathcal C_1) cl (\mathcal C_2) =
  cl (\mathcal C_1 \mathcal C_2)$, by Lemma~\ref{lemma:cl:cat}.  Given
  any admissible topology $\tau$, we see that the $\alpha$-products
  are closed in $\tau$, so $\tau$ is finer than the regular subword
  topology.
\end{proof}

There is a full, continuous map $\eta_Y \colon y \mapsto \dc y$ from
any space $Y$ to $\Sober Y$.  
The specialization preordering $\leq$ on $Y$
is characterized by $y \leq y'$ if and only if $y \in \eta_Y (y')$.
We note that $\eta_{X^{<\alpha}} (w)$ is simply $\overline w$.

\begin{thm}
  \label{thm:spec}
  Let $X$ be a Noetherian space, $\alpha$ be a bound.  Let $\ls$ be
  the specialization preordering of $X^{<\alpha}$.
  \begin{enumerate}
  \item For every $w \in X^{<\alpha}$, one can write $w$ as a finite
    concatenation of topologically indecomposable words
    $w_1 \cdots w_n$, and then
    $\eta_{X^{<\alpha}} (w) = \overline w = \overline {w_1} \cdots
    \overline {w_n}$, and
    $\overline {w_i} = {(\supp {w_i})}^{<|w_i|+1}$ for every $i$,
    $1\leq i\leq n$.
  \item For all transfinite words $w \eqdef w_1 \cdots w_m$ and
    $w' \eqdef w'_1 \cdots w'_n$ written as finite concatenations of
    topologically indecomposable words, $w \ls w'$ if and only if
    there are indices
    $0=i_0 \leq i_1 \leq \cdots \leq i_{n-1} \leq i_n = m$ such that
    for every $j$ with $1\leq j \leq n$,
    $\bigcup_{i=i_{j-1}+1}^{i_j} \supp w_i \subseteq \supp {w'_j}$ and
    $|w_{i_{j-1}+1}|+\cdots +|w_{i_j}| \leq |w'_j|$.
    %
  \end{enumerate}
\end{thm}
\begin{proof}
  (1) We write $w$ as $w_1 \cdots w_n$ where each $w_i$ is
  topologically indecomposable by Lemma~\ref{lemma:indecomp:decomp}.
  By Lemma~\ref{lemma:cl:cat} (and since
  $\overline\epsilon = \varepsilon$ in the base case $n=0$),
  $\overline w = \overline {w_1} \cdots \overline {w_n}$.  Finally,
  $\overline {w_i} = {(\supp {w_i})}^{<|w_i|+1}$ by
  Lemma~\ref{lemma:indecomp:cl}.

  (2) Lemma~\ref{lemma:indecomp:split} has the following consequence.
  Given any closed set of the form $\mathcal C_1 \mathcal C_2$ with
  $\mathcal C_1$ and $\mathcal C_2$ closed, for every transfinite word
  $w \eqdef w_1 \cdots w_m$ in $\mathcal C_1 \mathcal C_2$, written as
  a finite concatenation of topologically indecomposable words, there
  is an index $i$ with $0\leq i\leq m$ such that
  $w_1 \cdots w_i \in \mathcal C_1$ and
  $w_{i+1} \cdots w_m \in \mathcal C_2$.  Indeed, since
  $w \in \mathcal C_1 \mathcal C_2$, there is an index $j$ with
  $1\leq j \leq m$ such that one can write $w_j$ as $uv$, and
  $w_1 \cdots w_{j-1} u \in \mathcal C_1$,
  $v w_j \cdots w_m \in \mathcal C_2$.  Since $cat$ is continuous
  (Lemma~\ref{lemma:cat:cont}~(2)),
  $\mathcal C'_1 \eqdef cat (w_1 \cdots w_{j-1}, \_)^{-1} (\mathcal
  C_1)$ and
  $\mathcal C'_2 \eqdef cat (\_, w_j \cdots w_m)^{-1} (\mathcal C_2)$
  are closed.  Now $w_j = uv \in \mathcal C'_1 \mathcal C'_2$, so
  $w_j$ is in $\mathcal C'_1$ or in $\mathcal C'_2$; we define $i$ as
  $j$ in the first case, as $j-1$ in the second case.

  By induction on $n$, if $w$ belongs to a finite product
  $\mathcal C_1 \cdots \mathcal C_n$ of closed subsets of
  $X^{<\alpha}$ included in $X^{<\alpha}$, we can find indices
  $0=i_0 \leq i_1 \leq \cdots \leq i_{n-1} \leq i_n = m$ such that
  $w_{i_{j-1}+1} \cdots w_{i_j} \in \mathcal C_j$ for every $j$ with
  $1\leq j\leq n$.

  Let us assume that $w \ls w'$, namely
  $\overline w \subseteq \overline {w'}$.  By (1),
  $\overline {w_1} \cdots \overline {w_m} \subseteq \overline {w'_1}
  \cdots \overline {w'_n}$.  Letting
  $\mathcal C_j \eqdef \overline {w'_j}$, we obtain indices
  $0=i_0 \leq i_1 \leq \cdots \leq i_{n-1} \leq i_n = m$ such that
  $w_{i_{j-1}+1} \cdots w_{i_j} \in \overline {w'_j}$ for each $j$.
  Since $w'_j \in X^{<|w'_j|+1}$,
  $\overline {w'_j} \subseteq X^{<|w'_j|+1}$, so
  $|w_{i_{j-1}+1}|+\cdots +|w_{i_j}| = |w_{i_{j-1}+1} \cdots w_{i_j}|
  \leq |w'_j|$.  Also, for every $i$ with $i_{j-1}+1 \leq i \leq i_j$,
  $w_i \leq_* w_{i_{j-1}+1} \cdots w_{i_j}$, so $w_i$ is in
  $\overline {w'_j}$ by Lemma~\ref{lemma:prod}; hence
  $\supp {w_i} \subseteq \supp {w'_j}$, using
  Lemma~\ref{lemma:indecomp:cl} and Lemma~\ref{lemma:red:1}~(2).

  Conversely, if
  $\bigcup_{i=i_{j-1}+1}^{i_j} \supp w_i \subseteq \supp {w'_j}$ and
  $|w_{i_{j-1}+1}|+\cdots +|w_{i_j}| \leq |w'_j|$ for every $j$, then
  $w_{i_{j-1}+1} \cdots w_{i_j}$ is in
  ${(\supp {w'_j})}^{<|w'_j|+1} = \overline {w'_j}$ (by
  Lemma~\ref{lemma:indecomp:cl}); so $w = w_1 \cdots w_m$ is in
  $\overline {w'_1} \cdots \overline {w'_n} = \overline {w'}$, by (1).
\end{proof}
How does this compare to the subword ordering $\leq_*$?  By
Lemma~\ref{lemma:prod}, $w \leq_* w'$ implies $w \ls w'$.  The
converse holds on $X^* = X^{<\omega}$
\cite[Exercise~9.7.29]{JGL-topology}, and on $X^{<\omega+1}$ if $X$ is
a wqo in its Alexandroff topology
\cite[Proposition~5.16]{JGL:noethinf}.  Otherwise, the result may
fail, as the following demonstrates.
\begin{example}
  \label{ex:Xomega:ls}
  Let $X$ be $\nat$ with the cofinite topology.  Since its non-trivial
  closed subsets are finite, $X$ is Noetherian.
  In order to see that $\ls$ and $\leq_*$ differ, let
  $w \eqdef 0\;1\;2\; \cdots$.  For every non-empty suffix $v$ of $w$,
  $\Img v$ is infinite, so $\supp v = \nat$.  Hence $w$ is
  topologically indecomposable.  By Theorem~\ref{thm:spec}~(1),
  $\overline w=X^{<\omega+1}$, so $w' \ls w$ for every
  $w' \in X^{<\omega+1}$.  Hence, for
  example, $0^\omega \ls w$, but $0^\omega \not\leq_* w$.
\end{example}
For spaces $X^{<\alpha}$ with $\alpha \geq \omega^2$, we have the
following.
\begin{example}
  \label{ex:Xomega2:ls}
  Let $X$ be any Noetherian space with two elements $a$ and $b$ such
  that $a \not\leq b$ (say, $\{a, b\}$ with the discrete topology).
  We claim that $\ls$ and $\leq_*$ differ on $X^{<\alpha}$, for any
  bound $\alpha \geq \omega^2$.  Let $w' \eqdef (ab^\omega)^\omega$
  (explicitly, $w' \colon \omega^2 \to X$, $w' (\omega.m+n) \eqdef a$
  if $n=0$, $b$ otherwise); $w'$ is indecomposable, hence
  topologically indecomposable by Lemma~\ref{lemma:indecomp}, and
  $\supp {w'} = \dc \{a, b\}$.  Let $w \eqdef a^{\omega^2}$.  This is
  also indecomposable, and $\supp w = \dc a$.  Hence, by
  Theorem~\ref{thm:spec}, $w \ls w'$.  However, $w \not\leq_* w'$,
  since there is no strictly increasing map from
  $\omega^2$ 
  into $\omega$, the subset of positions of $w'$ where $a$ occurs.
\end{example}

\section{Conclusion and open problems}
\label{sec:concl-open-probl}

We have described a Noetherian topology on spaces of transfinite words
over a Noetherian space $X$, in particular on spaces $X^{<\alpha}$,
where $\alpha$ is a bound. In the latter situation, we have
characterized its irreducible closed subsets, and given upper bounds
on the stature and reduced sobrification rank of $X^{<\alpha}$. We have also
characterized the specialization preordering $\ls$ of $X^{<\alpha}$.

Although we have not stressed it, the syntax of $\alpha$-products
naturally yields an S-representation of $X^{<\alpha}$, in the sense of
\cite{FGL:compl,JGL:noethinf}, provided we restrict our bounds to lie
in some class of ordinals with a computable representation, decidable
ordering, and decidable equality.  S-representations are important in
forward analysis procedures for well-structured transition systems
\cite{FG-lmcs12}.

We finish with two questions.  First, it is frustrating that $\ls$ and
$\leq_*$ differ in general.  Is there a natural, finer
Noetherian topology on $X^{<\alpha}$ that 
would have $\leq_*$ as specialization preordering?  The specialization
preordering of a Noetherian space is necessarily well-founded.  Hence,
if the desired topology exists, $X^{<\alpha}$ is well-founded under
$\leq_*$.  By an argument similar to
Lemma~\ref{lemma:indecomp:decomp}, every non-empty word
$w \in X^{<\alpha}$ would have an indecomposable suffix, in other
words $X$ would have to be $\beta$-better-quasi-ordered in the sense
of \cite[Definition~IV-2]{Pouzet:pmo} or of \cite[Chapter~8,
5.1]{Fraisse:relations:I}, 
for every $\beta < \alpha$.

Second, what is the exact stature of $X^{<\alpha}$? its reduced
sobrification rank?

\normalsize
\bibliographystyle{abbrv}
\ifarxiv

\else%
\bibliography{noeth}
\fi%















\end{document}